\documentclass[11pt,english]{article}
\usepackage[T1]{fontenc}
\usepackage[latin9]{inputenc}
\usepackage[a4paper]{geometry}
\usepackage[active]{srcltx}
\usepackage{xcolor}
\usepackage{babel}
\usepackage{verbatim}
\usepackage{prettyref}
\usepackage{textcomp}
\usepackage{mathtools}
\usepackage{enumitem}
\usepackage{amsmath}
\usepackage{amsthm}
\usepackage{amssymb}
\usepackage{wasysym}
\usepackage{microtype}
\usepackage[unicode=true,pdfusetitle,
 bookmarks=true,bookmarksnumbered=true,bookmarksopen=false,
 breaklinks=false,pdfborder={0 0 1},backref=false,colorlinks=false]
 {hyperref}
\hypersetup{
 psdextra}

\makeatletter
\definecolor{darkblue}{rgb}{0,0,0.5}
\hypersetup{linkcolor=darkblue,urlcolor=blue,bookmarksopen=true,unicode}

\RequirePackage[nameinlink,nosort,noabbrev,capitalise]{cleveref}
\def\prettyref{\cref}
\theoremstyle{plain}
\newtheorem{thm}{\protect\theoremname}
\theoremstyle{plain}
\newtheorem{claim}[thm]{\protect\claimname}
\theoremstyle{plain}
\newtheorem*{remark*}{\protect\remarkname}
\theoremstyle{plain}
\newtheorem{lem}[thm]{\protect\lemmaname}
\def\numberedstatement{\begin{equation}\begin{minipage}[c][1\totalheight][t]{0.9\linewidth}}\def\endnumberedstatement{\end{minipage}\end{equation}}
\newenvironment{boldproof}[1][\proofname] {\par\pushQED{\qed}\normalfont\topsep6\p@\@plus6\p@\relax\trivlist\item[\hskip\labelsep\bfseries#1\@addpunct{.}]\ignorespaces}{\popQED\endtrivlist\@endpefalse}
\newenvironment{subproof}[1][\proofname]{%
\begin{proof}[#1]%
}{%
\end{proof}%
}
\providecommand*{\partitle}[1]{\textbf{#1}}
\makeatletter
\newcommand\present[1]{%
\write\@auxout{\noexpand\expandafter\noexpand\gdef\noexpand\csname present@#1\noexpand\endcsname{}}%
\expandafter\gdef\csname stillpresent@#1\endcsname{}%
}
\RequirePackage{calc}
\RequirePackage{xargs}

\newlength\sidebysidesep
\newcommand\sidebysideedgesepfactor{0}

\makeatletter
\newcommand\@eoaddto[2]{%
\expandafter\expandafter\expandafter\long\expandafter\expandafter\expandafter\def\expandafter\expandafter\expandafter#1\expandafter\expandafter\expandafter{\expandafter#1\expandafter{#2}}}
\newcommandx\@eosubstdef[9][1={},2={},3={},4={},5={},6={},7={}]{%
\def\@args{}%
\@eoaddto\@args{#1}%
\@eoaddto\@args{#2}%
\@eoaddto\@args{#3}%
\@eoaddto\@args{#4}%
\@eoaddto\@args{#5}%
\@eoaddto\@args{#6}%
\@eoaddto\@args{#7}%
\long\def\def@command##1##2##3##4##5##6##7{\long\def#8{#9}}%
\expandafter\def@command\@args%
}
\newcommandx\@esubstdef[9][1={},2={},3={},4={},5={},6={},7={}]{%
\edef\@args{{#1}{#2}{#3}{#4}{#5}{#6}{#7}}%
\long\def\def@command##1##2##3##4##5##6##7{\long\def#8{#9}}%
\expandafter\def@command\@args%
}
\newcommand\sidebysidebox[1]{\GenericError{}{\noexpand\sidebysidebox can only be used inside the \noexpand\sidebyside command}{}{}}
\newcounter{sidebysidebox@count}
\newcommandx\sidebyside[3][1=\sidebysideedgesepfactor,2=\sidebysidesep,usedefault]{%
\setcounter{sidebysidebox@count}{0}%
\def\sidebyside@first@hook{}%
\long\def\sidebysidebox##1{\@eosubstdef[\sidebyside@first@hook]\sidebyside@first@hook{####1\sidebysidebox@second{##1}}\stepcounter{sidebysidebox@count}}%
#3\relax%
\def\sidebyside@second@hook{\noindent\hspace*{0pt plus 1fil}\hspace*{0pt plus #1fill}}%
\@esubstdef[\thesidebysidebox@count]{\sidebysidebox@second####1}{\@eosubstdef[\sidebyside@second@hook]\sidebyside@second@hook{########1\sidebyside@box{##1}{####1}}}%
\sidebyside@first@hook%
\long\def\sidebyside@box##1##2{\begin{minipage}[t]{(\linewidth - (#2)*((##1)-1)-(#2)*\real{#1}*2) / (##1)}##2\relax\end{minipage}\hspace*{\fill}}%
\par\sidebyside@second@hook\hspace*{0pt plus -1fill}\hspace*{0pt plus #1fill}\hspace*{0pt plus 1fil}\par}

\date{}

\providecommand\phantomsection{}

\usepackage{parskip}
\usepackage{etoolbox}
\makeatletter
\patchcmd\deferred@thm@head
  {\addvspace{-\parskip}}
  {}
  {}{\typeout{\string\deferred@thm@head patch failed!}}
\usepackage{tikz}\usetikzlibrary{arrows}
\usepackage[justification=centerlast,font=small]{caption}
\usepackage{mleftright}
\mleftright

\RequirePackage{xargs}
\makeatletter
\newif\if@presentchanged
\def\presentchanged@warning{\if@presentchanged\else\typeout{LaTeX Warning: Label(s) may have changed. Rerun to get cross-references right.}\@presentchangedtrue\fi}
\providecommandx\ifpresent[3][3=]{\ifcsname present@#1\endcsname%
\AtEndDocument{\ifcsname stillpresent@#1\endcsname\else\presentchanged@warning\fi}\def\ifpresent@command{#2}%
\else \AtEndDocument{\ifcsname stillpresent@#1\endcsname\presentchanged@warning\fi}\def\ifpresent@command{#3}%
\fi%
\ifpresent@command%
}

\usepackage{refcount}
\usepackage{fmtcount}
\usepackage{crossreftools}

\usepackage{tocbibind}
\usepackage[numbers,sort]{natbib}
\g@addto@macro\bfseries{\boldmath}

\def\padded#1{\leavevmode \setbox \@tempboxa \hbox {\color@begingroup {\ensuremath{#1}}\color@endgroup }\@tempdima \fboxsep \advance \@tempdima \dp \@tempboxa  \hbox {\lower \@tempdima \hbox {\vbox {\vskip \fboxsep \box \@tempboxa \vskip \fboxsep }\relax}}}

\makeatother

\providecommand{\claimname}{Claim}
\providecommand{\lemmaname}{Lemma}
\providecommand{\remarkname}{Remark}
\providecommand{\theoremname}{Theorem}

\begin{document}
\global\long\def\naturals{\mathbf{N}}%
\global\long\def\integers{\mathbf{Z}}%
\global\long\def\reals{\mathbf{R}}%
\global\long\def\B{{\color{blue}\mathcal{B}}}%
\global\long\def\R{{\color{red}\mathcal{R}}}%
\global\long\def\Q{\mathcal{Q}}%
\global\long\def\tbd{{\color{magenta}??}}%
\global\long\def\tif{\text{if }}%
\global\long\def\tand{\text{ and }}%
\global\long\def\tandfinal{\text{ and}}%
\global\long\def\totherwise{\text{otherwise}}%
\global\long\def\A{\mathcal{A}}%
\global\long\def\C{\mathcal{C}}%
\global\long\def\F{\mathcal{F}}%
\global\long\def\fail{\textnormal{\ensuremath{\frownie}}}%
\global\long\def\Im{\operatorname{Im}}%
\global\long\def\Dom{\operatorname{Dom}}%
\global\long\def\mathcomment#1{}%
\global\long\def\Padded#1{\padded{#1}}%

\makeatletter
\global\long\def\labelifpresent#1{\ifpresent{#1}{}[\nonumber\@gobbletwo]}%
\makeatother
\title{Ramsey numbers of Boolean lattices}
\author{D\'aniel Gr\'osz\thanks{Department of Mathematics, University of Pisa, Pisa. e-mail: \protect\href{mailto:groszdanielpub@gmail.com}{groszdanielpub@gmail.com}}
\and  Abhishek Methuku\thanks{School of Mathematics, University of Birmingham, Birmingham. e-mail: \protect\href{mailto:abhishekmethuku@gmail.com}{abhishekmethuku@gmail.com}}
\and  Casey Tompkins\thanks{Discrete Mathematics Group, Institute for Basic Science (IBS), Daejeon. e-mail: \protect\href{mailto:ctompkins496@gmail.com}{ctompkins496@gmail.com}}}
\maketitle
\begin{abstract}
\noindent \setlength\parskip\medskipamount The \emph{poset Ramsey number} $R(\Q_{m},\Q_{n})$ is the smallest integer $N$
such that any blue-red coloring of the elements of the Boolean lattice $\Q_{N}$ has a blue induced copy of~$\Q_{m}$ or
a red induced copy of~$\Q_{n}$. The \emph{weak poset Ramsey number} $R_{w}(\Q_{m},\Q_{n})$ is defined analogously, with
weak copies instead of induced copies. It is easy to see that $R(\Q_{m},\Q_{n})\ge R_{w}(\Q_{m},\Q_{n})$.

\noindent Axenovich and Walzer~\cite{AxenovichWalzer} showed that $n+2\le R(\Q_{2},\Q_{n})\le2n+2$. Recently, Lu and Thompson~\cite{LuThompson}
improved the upper bound to $\frac{5}{3}n+2$. In this paper, we solve this problem asymptotically by showing that $R(\Q_{2},\Q_{n})=n+O(n/\log n)$.

\noindent In the diagonal case, Cox and Stolee~\cite{CoxStolee} proved $R_{w}(\Q_{n},\Q_{n})\ge2n+1$ using a probabilistic
construction. In the induced case, Bohman and Peng~\cite{BohmanPeng} showed $R(\Q_{n},\Q_{n})\ge2n+1$ using an explicit
construction. Improving these results, we show that $R_{w}(\Q_{m},\Q_{n})\ge n+m+1$ for all $m\ge2$ and large~$n$ by
giving an explicit construction; in particular, we prove that $R_{w}(\Q_{2},\Q_{n})=n+3$.
\end{abstract}

\section{Introduction}

\partitle{Background and definitions.} The classical Ramsey theorem asserts that for any $m$ and $n$, there is an integer~$N$
such that every blue-red edge coloring of the complete graph on $N$ vertices contains a blue clique on $m$ vertices or
a red clique on $n$ vertices. Determining the smallest such integer $N$, known as the Ramsey number is a central problem
in combinatorics. More generally, for any two graphs $G$ and $H$, the Ramsey number is the smallest integer $N$ such that
every blue-red edge coloring of the complete graph on $N$ vertices contains a red copy of $G$ or a blue copy of $H$. Several
natural variations of these problems such as multicolor Ramsey numbers, and hypergraph Ramsey numbers are major subjects
of ongoing research. For further examples, we refer the reader to the surveys ~\cite{ConlonFoxSudakov,MubayiSuk}.

In this paper, we will study poset Ramsey numbers. A \emph{partially ordered set} (or a \emph{poset} for short) is a set
with an accompanying relation $\le$ which is transitive, reflexive, and antisymmetric. A \emph{Boolean lattice} of dimension~$n$,
denoted by $\Q_{n}$, is the power set of $[n]\coloneqq\{1,2,\ldots,n\}$ equipped with the inclusion relation. If $(P,\le)$
and $(Q,\le')$ are posets, then an injection $f:P\to Q$ is \emph{order-preserving} if $f(x)\le'f(y)$ whenever $x\le y$;
we say that $f(P)$ is a \emph{weak copy} of~$P$ in~$Q$ and that $P$ is a \emph{weak subposet} of~$Q$. An injection
$f:P\to Q$ is an \emph{order-embedding} if $f(x)\le'f(y)$ if and only if $x\le y$; we say that $f(P)$ is an \emph{induced
copy} of $P$ in $Q$ and that $P$ is an \emph{induced subposet} of~$Q$.

For posets $P_{1}$ and $P_{2}$, the \emph{(induced) poset Ramsey number} $R(P_{1},P_{2})$ is defined to be the smallest
integer $N$ such that every blue-red coloring of the elements of the Boolean lattice $\Q_{N}$ contains an induced copy
of $P_{1}$ whose elements are blue or an induced copy of~$P_{2}$ whose elements are red. Similarly, the \emph{weak poset
Ramsey number} $R_{w}(P_{1},P_{2})$ is defined to be the smallest integer $N$ such that every blue-red coloring of the
elements of the Boolean lattice $\Q_{N}$ contains a weak copy of~$P_{1}$ whose elements are blue or a weak copy of~$P_{2}$
whose elements are red. (For convenience, we will call a copy of poset $P$ all of whose elements are blue is called a blue
copy of $P$, and a copy of poset~$P$ all of whose elements are red is called a red copy of~$P$.) It is easy to see that
$R(P_{1},P_{2})\ge R_{w}(P_{1},P_{2})$. The focus of this paper is the natural problem when $P_{1}$ and $P_{2}$ are Boolean
lattices $\Q_{m}$ and $\Q_{n}$ for $m,n\in\naturals$. Recently, variants of this problem, such as rainbow poset Ramsey
numbers have been studied in~\cite{Chang.etal,Chenetal,CoxStolee}.

\partitle{Induced poset Ramsey numbers.} For the diagonal poset Ramsey number $R(\Q_{n},\Q_{n}),$ Axenovich and Walzer~\cite{AxenovichWalzer}
showed that $2n\le R(\Q_{n},\Q_{n})\le n^{2}+2n$. Walzer~\cite{WalzerThesis} improved the upper bound to $R(\Q_{n},\Q_{n})\le n^{2}+1$.
Recently, Lu and Thompson~\cite{LuThompson} further improved it to $R(\Q_{n},\Q_{n})\le n^{2}-n+2$. On the other hand,
Cox and Stolee~\cite{CoxStolee} showed that for $n\ge13$, $R_{w}(\Q_{n},\Q_{n})\ge2n+1$, which implies that $R(\Q_{n},\Q_{n})\ge2n+1$.

More generally, Axenovich and Walzer~\cite{AxenovichWalzer} showed that $n+m\le R(\Q_{m},\Q_{n})\le mn+n+m$ for any integers
$n,m\ge1$. Lu and Thompson~\cite{LuThompson} improved this bound by showing that $R(\Q_{m},\Q_{n})\le(m-2+\frac{9m\text{\textminus}9}{(2m\text{\textminus}3)(m+1)})n+m+3$
for all $n\ge m\ge4$. See \cite{AxenovichWalzer,CoxStolee,LuThompson,WalzerThesis} for several other interesting results.

For the off-diagonal poset Ramsey number $R(\Q_{2},\Q_{n})$, Axenovich and Walzer~\cite{AxenovichWalzer} showed that $n+2\le R(\Q_{2},\Q_{n})\le2n+2$.
Recently, Lu and Thompson~\cite{LuThompson} improved the upper bound by proving that $R(\Q_{2},\Q_{n})\le\frac{5}{3}n+2$.
In this paper, we determine $R(\Q_{2},\Q_{n})$ asymptotically by proving the following theorem.
\begin{thm}
\label{thm:Upperbound}For every $c>2$, there exists an integer $n_{0}$ such that for all $n\ge n_{0}$, we have 
\[
R(\Q_{2},\Q_{n})\le n+c\frac{n}{\log_{2}n}.
\]
\end{thm}

Combining \prettyref{thm:Upperbound} with the lower bound $R(\Q_{2},\Q_{n})\ge n+2$, we obtain that $R(\Q_{2},\Q_{n})$
is asymptotically equal to~$n$. We prove \prettyref{thm:Upperbound} in \prettyref{sec:upper}. In fact, it follows from
our proof of \prettyref{thm:Upperbound} that for all $n\ge2$, we have $R(\Q_{2},\Q_{n})\le n+6.14\frac{n}{\log_{2}n}$.

\partitle{Weak poset Ramsey numbers.} A chain of length~$k$ is a poset of~$k$ distinct, pairwise comparable elements
and is denoted by~$C_{k}$. Cox and Stolee~\cite{CoxStolee} showed that $R_{w}(C_{k},\Q_{n})=n+k-1$; since $\Q_{m}$~is
a weak subposet of~$C_{2^{m}}$, this implies that $R_{w}(\Q_{m},\Q_{n})\le n+2^{m}-1$. The lower bound $R_{w}(\Q_{m},\Q_{n})\ge m+n$
is obtained by a simple ``layered'' coloring of~$\Q_{m+n-1}$ considered by Axenovich and Walzer~\cite{AxenovichWalzer},
which is described as follows. The collection of all subsets of~$[N]$ of a given size~$k$ is called a \emph{layer}. A
coloring of~$\Q_{N}$ is \emph{layered} if for every layer, all sets on that layer have the same color. A layered coloring
of~$\Q_{m+n-1}$ with $m$~blue layers and $n$~red layers does not contain a (weak) blue copy of~$\Q_{m}$ or a (weak)
red copy of~$\Q_{n}$. Therefore, $R_{w}(\Q_{m},\Q_{n})\ge m+n$ (which implies $R(\Q_{m},\Q_{n})\ge m+n$). Despite the
work of several researchers, so far this lower bound on $R_{w}(\Q_{m},\Q_{n})$ has not been improved except in the diagonal
case: Cox and Stolee~\cite{CoxStolee} showed that $R_{w}(\Q_{n},\Q_{n})\ge2n+1$ for $n\ge13$ using a probabilistic construction.
Recently, in the induced case, Bohman and Peng \cite{BohmanPeng} gave an explicit construction showing the bound $R(\Q_{n},\Q_{n})\ge2n+1$.
Note that these constructions showing $R(\Q_{n},\Q_{n})\ge2n+1$ cannot be layered.

We give an explicit construction which yields a lower bound on $R_{w}(\Q_{m},\Q_{n})$ for all~$m$ and $n\ge68$, thereby
generalizing the results of Bohman and Peng to the weak poset case, and additionally extending their results and those of
Cox and Stolee to the off-diagonal case.
\begin{thm}
\label{thm:Lowerbound}For any $m\ge2$ and $n\ge68$, we have \textup{
\[
R_{w}(\Q_{m},\Q_{n})\ge m+n+1.
\]
}
\end{thm}

Note that \prettyref{thm:Lowerbound} shows that $R_{w}(\Q_{2},\Q_{n})=n+3$ since $R_{w}(\Q_{2},\Q_{n})\le n+2^{2}-1=n+3$
by the upper bound mentioned earlier.

We prove \prettyref{thm:Lowerbound} in \prettyref{subsec:ours-general}. The construction and the proof of \prettyref{thm:Lowerbound}
are simpler if we restrict ourselves to the case of $m=2$ and consider induced subposets rather than weak subposets . Therefore,
in order to illustrate the main ideas of our construction, we present a short proof showing the special case $R(\Q_{2},\Q_{n})\ge n+3$
(for $n\ge18$) in \prettyref{subsec:ours-induced}. We also give a probabilistic construction for $m\ge3$ and $n$ sufficiently
large in \prettyref{subsec:cox stolee} by generalizing a construction of Cox and Stolee~\cite{CoxStolee}.

\section{\label{sec:upper}Upper bound: Proof of \texorpdfstring{\prettyref{thm:Upperbound}}{Theorem~\ref{thm:Upperbound}}}

Let $k=\left\lfloor c\frac{n}{\log_{2}n}\right\rfloor $. Assume that $\B,\R\subset\Q_{n+k}$ such that $\B\sqcup\R=\Q_{n+k}$,
and further assume that $\Q_{2}$~is not an induced subposet of $\B$, and $\Q_{n}$~is not an induced subposet of~$\R$.

Before continuing with the proof of \prettyref{thm:Upperbound}, let us provide an outline of the proof.

\emph{Outline of the proof.} We attempt to define an order-embedding $\varphi$ from $\Q_{n}$ into~$\R$ recursively, starting
with~$\emptyset$, in such a way that the image of each set only depends on the images of its proper subsets. For every
$A\subseteq[n]$, $\varphi(A)$ will be a superset of~$A$, possibly containing some additional elements from $[n+k]\setminus[n]$.

If $\emptyset\in\R$, then we set $\varphi(\emptyset)=\emptyset$. More generally, in order for $\varphi$ to be order-preserving,
for any set $A\in\Q_{n}$, $\varphi(A)$ must be a superset of the images of all proper subsets of~$A$; as long as the
minimal set that is a superset of $A$ and also has this property is in~$\R$, we set it as $\varphi(A)$. If instead this
minimal set is in~$\B$, then we proceed to add elements of $[n+k]\setminus[n]$ to it, in an order determined by some arbitrary
permutation~$\pi$ of $[n+k]\setminus[n]$, until we obtain a set that is in $\R$. Throughout this recursive procedure,
in addition to the injection~$\varphi$, we construct a function~$\alpha$ where $\alpha(A)$ records the number of elements
of $[n+k]\setminus[n]$ we need to include in~$\varphi(A)$ as a result of hitting sets in~$\B$ while attempting to embed
$A$ (and its subsets, during previous steps of the recursion); and another function~$f$, where $f(A)$ records an actual
chain of length~$\alpha(A)$, consisting of sets in~$\B$ that we have encountered while trying to embed $A$ and its subsets.

For any fixed permutation~$\pi$ of $[n+k]\setminus[n]$, the above embedding procedure can only fail if, at some point,
as we try to define $\varphi(A)$ for some $A\in\Q_{n}$, we hit a set in~$\B$, but we have already ``used up'' all $k$
elements of $[n+k]\setminus[n]$, so there are no elements left to add. In this event, we obtain a chain of length $k+1$,
contained in~$\B$. As $\Q_{n}$~is not an induced subposet of $\R$, the procedure must fail for all $k!$ permutations~$\pi$
of $[n+k]\setminus[n]$. This way, we can obtain a chain of length $k+1$ inside~$\B$, corresponding to each of these permutations.
We show that these $k!$ chains must all be distinct. We then show that the existence of $k!$ distinct chains of length
$k+1$ inside $\B$ implies that $\Q_{2}$ is an induced subposet of $\B$, a contradiction.

Now we continue with the proof of \prettyref{thm:Upperbound}.

At the core of the proof is \prettyref{lem:functions}. We will use the following notation: for a chain of sets $\mathcal{C}$
in~$\mathcal{\Q}_{n+k}$ of length~$l$, we denote its sets by $\left(q_{0},q_{1},\ldots,q_{l-1}\right)$ where $q_{0}\subseteq q_{1}\subseteq\ldots\subseteq q_{l-1}$.
\begin{claim}
\label{lem:functions}Let $\pi:[n+k]\setminus[n]\rightarrow[n+k]\setminus[n]$ be a permutation. There exist $\varphi:\Q_{n}\rightarrow\R\cup\{\fail\}$
(where $\fail$~is an arbitrary element, distinct from the members of~$\R$, and used solely to indicate failure to produce
an induced map into~$\R$), $\alpha:\Q_{n}\rightarrow\{0,1,\ldots,k,k+1\}$ and $f:\Q_{n}\rightarrow\C^{\le k+1}(\B)$,
where $\C^{\le k+1}(\B)$ is the family of all chains of length at most $k+1$ in~$\B$, with the following properties:\gdef\labelwidthi{\widthof{\textbf{\textup{L0. }}}}
\gdef\labeli{\textbf{\textup{P\arabic{enumi}. }}}
\gdef\propertyref#1{P#1}
\gdef\refi{\propertyref{\arabic{enumi}}}
\begin{enumerate}[label=\labeli, ref=\refi, labelsep=0em, leftmargin=0em, labelwidth=\labelwidthi, itemindent=\labelwidth, align=left]
\item \label{enu:phi containment}If $B,A\in\Q_{n}$ and $\varphi(B),\varphi(A)\in\R$, then $B\subsetneqq A\Longleftrightarrow\varphi(B)\subsetneqq\varphi(A)$.
(This implies that if $\fail\notin\Im\varphi$, then $\Q_{n}$~is an induced subposet of~$\R$.)
\item \label{enu:alpha order}If $B\subseteq A\in\Q_{n}$, then $\alpha(B)\le\alpha(A)$.
\item \label{enu:alpha phi}If $\alpha(A)=k+1$, then $\varphi(A)=\fail$. Otherwise $\varphi(A)\cap[n]=A$, and $\varphi(A)=A\cup\{\pi(n+\nobreak1),\pi(n+2),\ldots,\pi(n+\alpha(A))\}$.
\item \label{enu:f chain}For every $A\in\Q_{n}$, $f(A)=\left(f(A)_{0},f(A)_{1},\ldots,f(A)_{\alpha(A)-1}\right)$ is a chain
in~$\B$ of length $\alpha(A)$ with the property that $f(A)_{i}\setminus[n]=\{\pi(n+1),\pi(n+2),\ldots,\pi(n+i)\}$.
\item \label{enu:f subset}If $A\in\Q_{n}$ such that $1\le\alpha(A)\le k$, then $f(A)_{\alpha(A)-1}\subseteq\varphi(A)$. (In
fact this implies that $f(A)_{\alpha(A)-1}\subsetneqq\varphi(A)$, since the elements of~$f(A)$ are in~$\B$, while $\varphi(A)$
is in~$\R$. We do not use this observation.)
\end{enumerate}
\gdef\lastproperty{\propertyref{\arabic{enumi}}}
\end{claim}

\begin{proof}
We construct the functions $\varphi$, $\alpha$ and~$f$ recursively, and simultaneously prove the above properties by
induction: we set the values of these functions on a set $A\in\Q_{n}$ in such a way that they only depend on the values
of the functions on proper subsets of~$A$. (This includes the case of $A=\emptyset$ where no proper subsets exist, which
we do not treat in a special way for most of the proof. One can also consider the proof as a recursion and induction on the
size of the set $A$.) Let us fix an $A\in\Q_{n}$. Now we will define the values $\varphi(A)$, $\alpha(A)$ and $f(A)$,
and then prove that \propertyref{1} to \lastproperty{} hold for this set~$A$ under the assumption that they hold for
every proper subset of~$A$.

If there exists a $B\subsetneqq A$ such that $\varphi(B)=\fail$, then we pick such a set~$B$ arbitrarily, and set $\varphi(A)=\fail$,
$\alpha(A)=k+1$ and $f(A)=f(B)$. Otherwise let 
\begin{align*}
\beta & =\min\bigl\{ i\in\{0,1,\ldots,k\}:\left(\forall B\subsetneqq A:\alpha(B)\le i\right)\bigr\}\\
 & =\begin{cases}
{\displaystyle \max_{{\scriptscriptstyle B\subsetneqq A}}\alpha(B)} & \tif A\ne\emptyset,\\
0 & \tif A=\emptyset,
\end{cases}
\end{align*}
 and let 
\[
C=A\cup\{\pi(n+1),\pi(n+2),\ldots,\pi(n+\beta)\}=A\cup\left(\bigcup_{B\subsetneqq A}\varphi(B)\right)
\]
 (note that $\{\pi(n+1),\pi(n+2),\ldots,\pi(n+\beta)\}=\emptyset$ if $\beta=0$). We get the last equality by applying \ref{enu:alpha phi}
to the proper subsets of~$A$. We want $\varphi(A)$ to be a superset of $C$. If $C\in\R$, we set $\varphi(A)=C$. If
$C\in\B$, we keep adding $\pi(n+\beta+1),\pi(n+\beta+2),\ldots$ to it, until the set is not in~$\B$, if possible. That
is, let 
\[
\alpha(A)=\begin{cases}
\min\left\{ \begin{gathered}i\in\{\beta,\beta+1,\ldots,k\}:\\
C\cup\{\pi(n+\beta+1),\pi(n+\beta+2),\ldots,\pi(n+i)\}\in\R
\end{gathered}
\right\}  & \text{if such \ensuremath{i} exists,}\\
k+1 & \totherwise.
\end{cases}
\]
Then let 
\[
\varphi(A)=\begin{cases}
C\cup\{\pi(n+\beta+1),\pi(n+\beta+2),\ldots,\pi(n+\alpha(A))\} & \tif\alpha(A)\le k,\\
\fail & \tif\alpha(A)=k+1.
\end{cases}
\]
 Note that in the first case, 
\begin{gather*}
C\cup\{\pi(n+\beta+1),\pi(n+\beta+2),\ldots,\pi(n+\alpha(A))\}\\
=A\cup\{\pi(n+1),\pi(n+2),\ldots,\pi(n+\alpha(A))\}.
\end{gather*}
Furthermore if $A=\emptyset$, set $f(A)=()$, an empty chain. Otherwise pick a set $B\subsetneqq A$ such that $\alpha(B)=\beta$.
We set $f(A)$ to be a chain of length $\alpha(A)$ in~$\B$: 
\[
f(A)_{i}=\begin{cases}
f(B)_{i} & \tif0\le i<\beta,\\
A\cup\{\pi(n+1),\pi(n+2),\ldots,\pi(n+i)\} & \tif\beta\le i<\alpha(A).
\end{cases}
\]

Note that these definitions of $\varphi(A)$, $\alpha(A)$ and~$f(A)$ only depend on the values of these functions for
proper subsets of~$A$, so our recursive definitions make sense. It is easy to check that the definitions of~$\varphi$
and~$\alpha$ satisfy \ref{enu:alpha order}~and~\ref{enu:alpha phi}, which together imply \ref{enu:phi containment}.
\ref{enu:f chain}~and~\ref{enu:f subset} are also trivially satisfied when $\alpha(A)=0$. If $\alpha(A)=\beta=k+1$,
we have defined $f(A)=f(B)$ for some $B\subsetneqq A$ such that $\varphi(B)=\fail$; then \ref{enu:f chain} follows because
it holds for~$B$ by induction, and \ref{enu:f subset} is trivial.

Now we prove \ref{enu:f chain} and \ref{enu:f subset} when $\alpha(A)>0$ and $\beta\le k$. In the case where $\alpha(A)=\beta$
(equivalently if $C\in\R$, $\varphi(A)=C$ and $\alpha(A)=\alpha(B)$), then $f(A)=f(B)$ is a chain satisfying \ref{enu:f chain}
by induction. Since \ref{enu:f subset} holds for~$B$ by induction, we get that $f(A)_{\alpha(A)-1}=f(B)_{\alpha(B)-1}\subset\varphi(B)\subset\varphi(A)$,
so \ref{enu:f subset} is satisfied for~$A$ as well. If $\alpha(A)>\beta$, then \ref{enu:f subset} follows from the definitions
of $f(A)$ and~$\varphi(A)$. Furthermore, $A\cup\{\pi(n+1),\pi(n+2),\ldots,\allowbreak\pi(n+i)\}\in\B$ for $\beta\le i<\alpha(A)$
because $\alpha(A)$~was chosen as the smallest~$i\ge\beta$ such that $A\cup\{\pi(n+1),\pi(n+2),\ldots,\pi(n+i)\}\in\R$;
this is enough to show \ref{enu:f chain} if~$\beta=0$. Finally, if $\alpha(A)>\beta>0$, then $f(A)$~is obtained by
concatenating the chains $f(B)$ and $\bigl(A\cup\{\pi(n+1),\pi(n+2),\ldots,\pi(n+i)\}\bigr)_{\beta\le i<\alpha(A)}$. By
induction $f(B)$~is a chain satisfying the conditions of \ref{enu:f chain}, and $B$ satisfies \ref{enu:f subset}, so
$f(B)_{\beta-1}\subset\varphi(B)$. Using that \ref{enu:alpha phi} holds for~$B$ by induction, we also have $\varphi(B)=B\cup\{\pi(n+1),\pi(n+2),\ldots,\pi(n+\alpha(B))\}\subsetneqq A\cup\{\pi(n+1),\allowbreak\pi(n+2),\ldots,\pi(n+\beta)\}$
(recall that $B\subsetneqq A$ and $\beta=\alpha(B)$). Thus $f(A)$~is indeed a chain satisfying \ref{enu:f chain}. The
proof of the properties \propertyref{1} to \lastproperty{} of $\varphi$, $\alpha$ and~$f$ is now complete.
\end{proof}
For an arbitrary permutation $\pi:[n+k]\setminus[n]\rightarrow[n+k]\setminus[n]$, let $\varphi^{\pi}$, $\alpha^{\pi}$
and $f^{\pi}$ be the maps given by \prettyref{lem:functions}. If $\Im\varphi^{\pi}\subseteq\R$, then $\varphi^{\pi}$~shows
that $\Q_{n}$~is an induced subposet of~$\R$ by \ref{enu:phi containment}. Assume that this is not the case. Then, for
some $A\in\Q_{n}$, $\varphi^{\pi}(A)=\fail$, $\alpha^{\pi}(A)=k+1$ by~\ref{enu:alpha phi}, and $f^{\pi}(A)$~is a chain
of length $k+1$ in~$\B$ by~\ref{enu:f chain}. By \ref{enu:f chain}, we have $\pi(n+i)=\left(f^{\pi}(A)_{i}\setminus f^{\pi}(A)_{i-1}\right)\setminus[n]$
when $1\le i\le\alpha^{\pi}(A)-1$, so if $\alpha^{\pi}(A)=k+1$, then one can recover the permutation $\pi$ from the chain~$f^{\pi}(A)$.

Under our assumption that $\Q_{n}$~is not an induced subposet of~$\R$, we get a distinct chain $f^{\pi}$ of length $k+1$
in~$\B$ for each of the $k!$ permutations~$\pi$ of $[n+k]\setminus[n]$, with the property that 
\[
\forall\,0\le i\le k:f_{i}^{\pi}\setminus[n]=\{\pi(n+1),\pi(n+2),\ldots,\pi(n+i)\}.
\]
We claim that the map $\pi\mapsto\left(f_{0}^{\pi},f_{k}^{\pi}\right)$ is injective. Let $\pi_{1}$ and~$\pi_{2}$ be two
different permutations of $[n+k]\setminus[n]$. Let $i=\min_{j\in\{0,\ldots,k\}}\pi_{1}(n+j)\neq\pi_{2}(n+j)$. Then $\pi_{1}(i)\in f_{i}^{\pi_{1}}$,
$\pi_{1}(i)\notin f_{i}^{\pi_{2}}$,$\pi_{2}(i)\in f_{i}^{\pi_{2}}$ and $\pi_{2}(i)\notin f_{i}^{\pi_{1}}$, so $f_{i}^{\pi_{1}}$
and $f_{i}^{\pi_{2}}$ are unrelated. So if $f_{0}^{\pi_{1}}=f_{0}^{\pi_{2}}$ and $f_{k}^{\pi_{1}}=f_{k}^{\pi_{2}}$, then
$\B$~would contain an induced copy of $\Q_{2}$, a contradiction.

Since the map $\pi\mapsto\left(f_{0}^{\pi},f_{k}^{\pi}\right)$ is injective, 
\[
k!\le\left(2^{n+k}\right)^{2}=2^{2(n+k)}.
\]
 Approximating the left-hand side:
\[
k!>\left(\frac{k}{e}\right)^{k}=2^{k(\log_{2}k-\log_{2}e)}\text{, so}
\]
\begin{equation}
k(\log_{2}k-\log_{2}e)<2(n+k).\label{eq:2(n+k)}
\end{equation}
 Since $k=\left\lfloor c\frac{n}{\log_{2}n}\right\rfloor $, 
\begin{equation}
k\log_{2}k>\left(c\frac{n}{\log_{2}n}-1\right)\left(\log_{2}c+\log_{2}n-\log_{2}\log_{2}n-1\right)=cn(1-o(1)).\label{eq:cn}
\end{equation}
Since $c>2$, \eqref{eq:cn} contradicts \eqref{eq:2(n+k)} for sufficiently large~$n$. This completes the proof of \prettyref{thm:Upperbound}.
\begin{remark*}
It follows the above proof that for all $n\ge2$, we have $R(\Q_{2},\Q_{n})\le\nobreak n+\nobreak6.14\frac{n}{\log_{2}n}$.
Here we give a sketch of the calculations.

For $c=6.14$, we have $k=\nobreak\bigl\lfloor6.14\frac{n}{\log_{2}n}\bigr\rfloor>5.611\frac{n}{\log_{2}n}$ for every integer
$n\ge2$, and therefore we have $k\log_{2}k>5.611\frac{n}{\log_{2}n}\allowbreak\bigl(\log_{2}n\bigl(1-\frac{\log_{2}\log_{2}n}{\log_{2}n}\bigr)+\log_{2}5.611\bigr)\ge2.977n+13.96\frac{n}{\log_{2}n}$.
Using \eqref{eq:2(n+k)}, it can be shown that $0.8797k\log_{2}k\le k(\log_{2}k-\log_{2}e)\overset{{\scriptscriptstyle \eqref{eq:2(n+k)}}}{<}2n+12.28\frac{n}{\log_{2}n}$
for every $n\ge2$, contradicting the lower bound on $k\log_{2}k$ shown earlier.
\end{remark*}

\section{\label{sec:lower}Lower bounds}

\subsection{\label{subsec:ours-induced}An explicit construction showing $R(\protect\Q_{2},\protect\Q_{n})\ge n+3$}

In this subsection, we prove a special case of \prettyref{thm:Lowerbound} to illustrate the basic ideas of the construction.
The fully general proof of \prettyref{thm:Lowerbound}, presented in \prettyref{subsec:ours-general}, is significantly more
involved (primarily due the fact that it is more difficult to deduce properties of a weak map $\Q_{n}\rightarrow\Q_{n+m}$).
\begin{thm}
\label{prop:m_is_2-1}For $n\ge18$, there exist $\B,\R\subset\Q_{n+2}$ such that $\B\sqcup\R=\Q_{n+2}$, $\Q_{2}$~is
not an induced subposet of~$\B$, and $\Q_{n}$~is not an induced subposet of~$\R$.
\end{thm}

Let $k=\left\lfloor \frac{n}{2}\right\rfloor $. Let $\B\supset\binom{[n+2]}{k}\cup\binom{[n+2]}{k+3}$, with some sets of
size $k+1$ which we will add later. Assume for a contradiction that $\Q_{n}$~is an induced subposet of~$\R$. Let $\varphi:\Q_{n}\rightarrow\R$
be an injection such that $\varphi(A)\subseteq\varphi(B)$ if and only if $A\subseteq B$.

For any maximal chain $\emptyset\subsetneqq A_{1}\subsetneqq\ldots\subsetneqq A_{n-1}\subsetneqq[n]$, the sets in its image
satisfy $\varphi(\emptyset)\subsetneqq\varphi(A_{1})\subsetneqq\ldots\subsetneqq\varphi(A_{n-1})\subsetneqq\varphi([n])$,
and none of the sets in the image are of size $k$ or $k+3$. So for every $A\subseteq[n]$, 
\begin{equation}
\left|\varphi(A)\right|=\begin{cases}
\left|A\right| & \tif\left|A\right|\le k-1,\\
\left|A\right|+1 & \tif k\le\left|A\right|\le k+1,\\
\left|A\right|+2 & \tif k+2\le\left|A\right|,
\end{cases}\label{eq:levels-1-1}
\end{equation}
 thus the image of every singleton is a singleton (and the image of the complement of every singleton is the complement of
a singleton).

For $a\in[n]$, let $\tilde{\varphi}(a)$ denote the unique element of $\varphi(\{a\})$. The map $\tilde{\varphi}:[n]\rightarrow[n+2]$
is an injection. Note that, for a set $A\subseteq[n]$, $\tilde{\varphi}[A]$ denotes the image of $A$ under $\tilde{\varphi}$,
and for a set $B\subseteq[n+m]$, $\tilde{\varphi}^{-1}[B]$ denotes the preimage of $B$ under $\tilde{\varphi}$.

Let $X=\left\{ \tilde{\varphi}(a):a\in[n]\right\} $ and $Y=[n+2]\setminus X=\{y,z\}$. We have $\left|X\right|=n$ and $\left|Y\right|=2$.
We claim that for every $A\subseteq[n]$, $\varphi(A)\cap X=\tilde{\varphi}[A]$. Indeed, for every $b\in X$, there is an
$a\in[n]$ such that $\tilde{\varphi}(a)=b$, and we have $b=\tilde{\varphi}(a)\in\tilde{\varphi}[A]\Longleftrightarrow a\in A\Longleftrightarrow\{a\}\subseteq A\Longleftrightarrow\{\tilde{\varphi}(a)\}=\varphi(\{a\})\subseteq\varphi(A)\Longleftrightarrow b=\tilde{\varphi}(a)\in\varphi(A)$.

From~\eqref{eq:levels-1-1}, $\varphi(A)$ contains neither $y$ nor $z$ if $\left|A\right|\le k-1$, exactly one of them
if $k\le\left|A\right|\le k+1$, and both if $k+2\le\left|A\right|$.
\begin{claim}
\label{claim:y-1-1}For every set $A\in\binom{[n]}{k}$, $\varphi(A)$ contains the same element of $Y$.
\end{claim}

\begin{proof}
Assume for a contradiction that some sets of the form $\varphi(A)$ contain $y$, and others contain $z$. Since the Johnson
graph -- whose vertices are $\binom{[n]}{k}$, and whose edges connect sets with symmetric difference $2$ -- is connected,
there would be two sets $A,B\in\binom{[n]}{k}$ with a symmetric difference of size $2$ such that $y\in\varphi(A)$ and
$z\in\varphi(B)$. Then $\left|A\cup B\right|=k+1$, and $\varphi(A\cup B)$ would contain both $y$ and $z$ as it would
have to be a superset of both $\varphi(A)$ and $\varphi(B)$, contradicting that it contains exactly one of $y$ and $z$.
\end{proof}
Now we specify which sets of $\binom{[n+2]}{k+1}$ are added to~$\B$ in addition to $\binom{[n+2]}{k}\cup\binom{[n+2]}{k+3}$.
Our goal is to add these sets in such a way that for every map $\varphi:\Q_{n}\to\Q_{n+2}$, assuming that $\Im\varphi\subseteq\R$,
and $\varphi$ is an order-embedding, the above observations lead to a contradiction. (The map~$\varphi$, and the variables
dependent on it such as $X$, $Y$, $y$ and $z$, are not fixed now; we have to set $\B$ in such a way that the existence
of any order-embedding $\varphi:\Q_{n}\to\R$ leads to a contradiction.) For every distinct~$y,z\in[n+2]$, pick a set $C_{y,z}\in\binom{[n+2]}{k+1}$
such that $y\in C_{y,z}$ but $z\notin C_{y,z}$, and $C_{y,z}$ is at a symmetric difference of size at least $4$ from
every previously chosen set $C_{y',z'}$, and add $C_{y,z}$ to $\B$. We can do this by greedily picking sets~$C_{y,z}$
one-by-one: At each step, we have picked at most $(n+2)(n+1)-1$ sets so far, and each previously picked set~$C_{y',z'}$
blocks at most $1+k(n-k)$ choices (because there are at most that many sets containing $y$ but not $z$ with a symmetric
difference of size at most~2 from~$C_{y',z'}$). In total, there are $\binom{n}{k}$ sets $C\in\binom{[n+2]}{k+1}$ that
satisfy $y\in C$ and $z\notin C$. For $n\ge18$ and $k=\left\lfloor \frac{n}{2}\right\rfloor $, we have 
\[
\binom{n}{k}>\bigl((n+2)(n+1)-1\bigr)\bigl(1+k(n-k)\bigr),
\]
so we can always choose a set $C_{y,z}$ which satisfies the required conditions.

After adding such sets $C_{y,z}$ for every distinct $y,z\in[n+2]$, the resulting family $\B$ will not contain a copy of
$\Q_{2}$. Indeed, a copy of $\Q_{2}$ in~$\B$ would have to consist of a set of size~$k$, a set of size $k+3$, and
two sets of size $k+1$; but the latter two sets would need to have a symmetric difference of size~$2$.

Now assume for a contradiction that $\R=2^{[n+2]}\setminus\B$ contains an induced copy of~$\Q_{n}$. Consider an arbitrary
injection $\varphi:\Q_{n}\rightarrow\R$, and define $\tilde{\varphi}$, $X$ and $Y=\{y,z\}$ as before, and apply \prettyref{claim:y-1-1}.
We can assume without loss of generality that for every $A\in\binom{[n]}{k}$, we have $y\in\varphi(A)$, and thus $\varphi(A)=\tilde{\varphi}[A]\cup\{y\}$.
There is a set $C_{y,z}\in\binom{[n+2]}{k+1}\cap\B$ such that $y\in C_{y,z}$, but $z\notin C_{y,z}$. We have $\tilde{\varphi}^{-1}[C_{y,z}]\in{[n] \choose k}$,
and 
\[
\varphi\left(\tilde{\varphi}^{-1}[C_{y,z}]\right)=\left\{ \tilde{\varphi}(a):a\in[n],\tilde{\varphi}(a)\in C_{y,z}\right\} \cup\{y\}=C_{y,z}\in\B,
\]
contradicting that the image of~$\varphi$ is in~$\R$.

\subsection{\label{subsec:ours-general}An explicit construction showing \textmd{\normalsize{}$R_{w}(\protect\Q_{m},\protect\Q_{n})\ge m+n+1$}}

\prettyref{thm:Lowerbound} will be an immediate consequence of the following \MakeLowercase{\crtcrefnamebylabel{prop:noninduced-lower-general}}.
\begin{thm}
\label{prop:noninduced-lower-general}Let $n,m\in\naturals$ such that $m\ge2$ and $n\ge\sqrt{32m+260}+18$. There exist
$\B,\R\subset\Q_{n+m}$ such that $\B\sqcup\R=\Q_{n+m}$, $\Q_{m}$~is not a weak subposet of~$\B$, and $\Q_{n}$~is
not a weak subposet of~$\R$.
\end{thm}

To see that \prettyref{thm:Lowerbound} follows from \prettyref{prop:noninduced-lower-general}, notice that we can assume
without loss of generality that $n\ge m$, so for $n\ge68$, we have $n\ge\sqrt{32n+260}+18$.

For values of $m$ less than~$67$, the threshold for~$n$ in the hypothesis of~\prettyref{prop:noninduced-lower-general}
is smaller than $68$: for instance, it holds for $m=2$ and $n\ge36$. Also note that in the proof of \prettyref{lem:modp code},
we use Bertrand's postulate to obtain a prime $N\le p<2(N-1)$. By finding the smallest prime greater than or equal to $N$,
one may be able to relax the requirement on $k$ in \prettyref{lem:modp code}, and thereby extend \prettyref{prop:noninduced-lower-general}
to somewhat smaller values of $n$, for a given $m$.

In the proof of \prettyref{prop:noninduced-lower-general}, we will use \prettyref{lem:modp code}, which will follow from
\prettyref{lem:modp subset}.
\begin{lem}[Olson \cite{olson}]
\label{lem:modp subset}Let $p$ be a prime, and let $A\subseteq[p]$ such that $\left|A\right|\ge\sqrt{4p-3}$. Then for
every $a\in\integers$, there is a subset $B\subseteq A$ such that 
\[
\sum B\equiv a\pmod p.
\]
\end{lem}

\begin{lem}
\label{lem:modp code}Let $N,k\in\naturals$ such that $N>3$ and $k\ge\sqrt{8N-15}$. Then there is a constant-weight
code $\C\subset\binom{[N]}{k+1}$ such that the symmetric difference between any two sets is of size at least $4$, and the
following holds:
\begin{numberedstatement}
Let $n,m\in\naturals$ such that $n+m=N$ and $k\le n-\sqrt{8N-15}$. For every $Y\in\binom{[N]}{m}$ and $y\in Y$, there
is a set $C\in\binom{[N]\setminus Y}{k}$ such that $C\cup\{y\}\in\C$.\label{eq:modp code statement}
\end{numberedstatement}
\end{lem}

\begin{proof}
By Bertrand's postulate%
, there is a prime~$p$ such that $N\le p<2(N-1)$. Let $d\in[p]$ be a fixed constant, and let 
\[
\C=\left\{ S\in\binom{[N]}{k+1}:\sum S\equiv d\pmod p\right\} .
\]

Let $Y\in\binom{[N]}{m}$ and $y\in Y$. We have to find a set~$C$ that satisfies the condition in the statement. Let $l=\left\lceil \sqrt{8N-15}\right\rceil $.
It follows from the conditions of the lemma that $n\ge2l$. Let $[N]\setminus Y=\{x_{1},\ldots,x_{n}\}$ such that $x_{1}<x_{2}<\ldots<x_{n}$.
For $i=1,\ldots,l$, let $a_{i}=x_{i}$ and $b_{i}=x_{n-i+1}$. The numbers $b_{i}-a_{i}$ are in $[p]$, and they are different
because $b_{1}-a_{1}>b_{2}-a_{2}>\ldots>b_{l}-a_{l}$. Let $a=\sum_{i=1}^{l}a_{i}$. Let $E$ be a subset of $\{x_{l+1},\ldots,x_{n-l}\}$
with $k-l$ elements, and let $e=\sum E$. (It follows from the conditions that $0\le k-l\le n-2l$.)

Since $l\ge\sqrt{8N-15}\ge\sqrt{4p-3}$, by \prettyref{lem:modp subset}, there is a subset of $\{b_{1}-a_{1},\ldots,b_{l}-a_{l}\}$
such that its sum is congruent with $d-y-e-a$. That is, there is a set $I\subseteq[l]$ such that 
\[
\sum_{i\in I}(b_{i}-a_{i})\equiv d-y-e-a\pmod p.
\]
 Let 
\[
C=E\cup\left\{ \Padded{\begin{cases}
a_{i} & \tif i\notin I\\
b_{i} & \tif i\in I
\end{cases}:i\in[l]}\right\} .
\]

We have 
\[
\sum\left(C\cup\{y\}\right)=\sum E+\sum_{i=1}^{l}a_{i}+\sum_{i\in I}(b_{i}-a_{i})+y\equiv d\pmod p,
\]
 so $C\in\C$.
\end{proof}
 Let $k$ be an integer between $\sqrt{8(n+m)-15}$ and $n-1-\sqrt{8(n+m)-15}$ inclusive. (The conditions of the proposition
imply that $\sqrt{8(n+m)-15}+1\le n-1-\sqrt{8(n+m)-15}$, therefore such an integer $k$ exists.) Let $\B\supset\binom{[n+m]}{k}\cup\binom{[n+m]}{k+3}\cup\binom{[n+m]}{k+4}\cup\ldots\cup\binom{[n+m]}{k+m+1}\cup\C$,
where $\C$~is given by \prettyref{lem:modp code}, using $n+m$ in the place of~$N$. First we show that the family $\B$
does not contain a~$\Q_{m}$. Indeed, any two sets in~$\B$ of size $k+1$ have a symmetric difference of size at least~$4$.
A copy of~$\Q_{m}$ in~$\B$ would consist of a set of size~$k$, some sets of size $k+3,\ldots,k+m+1$ corresponding
to the sets of size~$2$~to~$m$ of the~$\Q_{m}$, and $m$~sets of size $k+1$ corresponding to the singletons of~$\Q_{m}$.
The latter $m$~sets would need to have a symmetric difference of size~$2$.

Assume for a contradiction that $\Q_{n}$~is a subposet of~$\R$. Let $\varphi:\Q_{n}\rightarrow\R$ be an injection that
preserves relations.

For any maximal chain $\emptyset\subsetneqq A_{1}\subsetneqq\ldots\subsetneqq A_{n-1}\subsetneqq[n]$, we have $\varphi(\emptyset)\subsetneqq\varphi(A_{1})\subsetneqq\ldots\subsetneqq\varphi(A_{n-1})\subsetneqq\varphi([n])$,
and none of the sets in the image are of size $k$ or $k+3,k+4,\ldots,k+m+1$. So for every $A\subseteq[n]$, 
\begin{equation}
\left|\varphi(A)\right|=\begin{cases}
\left|A\right| & \tif\left|A\right|\le k-1,\\
\left|A\right|+1 & \tif k\le\left|A\right|\le k+1,\\
\left|A\right|+m & \tif k+2\le\left|A\right|,
\end{cases}\label{eq:levels}
\end{equation}
 thus the image of every singleton is a singleton, and the image of the complement of every singleton is the complement of
a singleton).

For $a\in[n]$, let $\varphi_{1}(a)$ denote the unique element of $\varphi(\{a\})$, and let $\varphi_{2}(a)$ denote the
unique element of $[n+m]\setminus\varphi\left([n]\setminus\{a\}\right)$. Note that, for a set $A\subseteq[n]$, $\varphi_{i}[A]$
denotes the image of $A$ under $\varphi_{i}$, and for a set $B\subseteq[n+m]$, $\varphi_{i}^{-1}[B]$ denotes the preimage
of $B$ under $\varphi_{i}$.

The maps $\varphi_{1}$ and $\varphi_{2}$ are injections. Furthermore, for any distinct $a,b\in[n]$, it holds that $\{a\}\subseteq[n]\setminus\{b\}$,
so $\{\varphi_{1}(a)\}=\varphi(\{a\})\subseteq\varphi\left([n]\setminus\{b\}\right)=[n+m]\setminus\{\varphi_{2}(b)\}$, so
$\varphi_{1}(a)\ne\varphi_{2}(b)$. Now take the sets $\left\{ \varphi_{1}(a),\varphi_{2}(a)\right\} \subseteq[n+m]$ for
each $a\in[n]$; these sets have 1~or~2 elements, depending on whether $\varphi_{1}(a)=\varphi_{2}(a)$. Based on the observations
in this paragraph, if $a\ne b$, we have
\[
\left\{ \varphi_{1}(a),\varphi_{2}(a)\right\} \cap\left\{ \varphi_{1}(b),\varphi_{2}(b)\right\} =\emptyset.
\]
 Since $\bigcup_{a\in[n]}\{\varphi_{1}(a),\varphi_{2}(a)\}\subseteq[n+m]$, we have $\sum_{a\in[n]}\left|\{\varphi_{1}(a),\varphi_{2}(a)\}\right|\le n+m$,
so the number of these sets which have 2~elements is at most~$m$; in other words, 
\[
\left|\left\{ a\in[n]:\varphi_{1}(a)\ne\varphi_{2}(a)\right\} \right|\le m.
\]

Let 
\begin{align*}
D & =\left\{ a\in[n]:\varphi_{1}(a)=\varphi_{2}(a)\right\} ,\\
E & =[n]\setminus D,\\
X_{12} & =\varphi_{1}[D]=\varphi_{2}[D],\\
X_{1} & =\varphi_{1}[E],\\
X_{2} & =\varphi_{2}[E],\\
X_{\emptyset} & =[n+m]\setminus(X_{12}\cup X_{1}\cup X_{2}).
\end{align*}
 Then 
\begin{align}
[n+m] & =X_{12}\sqcup X_{1}\sqcup X_{2}\sqcup X_{\emptyset},\label{eq:disjoint}\\
\Im\varphi_{1} & =X_{12}\cup X_{1},\\
\Im\varphi_{2} & =X_{12}\cup X_{2},\nonumber \\
\left|X_{12}\right| & =\left|D\right|\ge n-m,\labelifpresent{grey}\label{eq:D ge n-m}\\
\left|X_{1}\right| & =\left|X_{2}\right|=\left|E\right|\le m,\nonumber \\
\left|X_{12}\right|+\left|X_{1}\right| & =\left|D\right|+\left|E\right|=n,\nonumber \\
\left|E\right|+\left|X_{\emptyset}\right| & =(n+m)-\left(\left|X_{12}\right|+\left|X_{1}\right|\right)=m.\label{eq:e-x0}
\end{align}

We have that, for every $A\subseteq[n]$, 
\begin{equation}
\forall a\in A:\varphi_{1}(a)\in\varphi(A)\label{eq:a in A}
\end{equation}
 because $a\in A\Rightarrow\{a\}\subseteq A\Rightarrow\{\varphi_{1}(a)\}=\varphi(\{a\})\subseteq\varphi(A)\Rightarrow\varphi_{1}(a)\in\varphi(A)$.
(Equivalently, $\varphi_{1}[A]\subseteq\varphi(A)$.) Symmetrically, 
\begin{equation}
\forall a\in[n]\setminus A:\varphi_{2}(a)\notin\varphi(A)\label{eq:a notin A}
\end{equation}
 because $a\in[n]\setminus A\Rightarrow A\subseteq[n]\setminus\{a\}\Rightarrow\varphi(A)\subseteq\varphi\left([n]\setminus\{a\}\right)=[n+m]\setminus\{\varphi_{2}(a)\}\Rightarrow\varphi_{2}(a)\notin\varphi(A)$.

For an $A\subseteq[n]$, let 
\[
F(A)=\left\{ \Padded{\begin{cases}
\varphi_{2}(a) & \tif a\in A\\
\varphi_{1}(a) & \tif a\notin A
\end{cases}:a\in E}\right\} \cup X_{\emptyset}.
\]
 By \eqref{eq:disjoint}, \eqref{eq:a in A} and \eqref{eq:a notin A}, we have 
\begin{equation}
\begin{aligned}\varphi(A)\cap\left([n+m]\setminus F(A)\right) & =\varphi(A)\cap\left(X_{12}\cup\left\{ \Padded{\begin{cases}
\varphi_{1}(a) & \tif a\in A\\
\varphi_{2}(a) & \tif a\notin A
\end{cases}:a\in E}\right\} \right)\\
 & =\varphi(A)\cap\left\{ \Padded{\begin{cases}
\varphi_{1}(a) & \tif a\in A\\
\varphi_{2}(a) & \tif a\notin A
\end{cases}:a\in[n]}\right\} \overset{\eqref{eq:a in A},\eqref{eq:a notin A}}{=}\varphi_{1}[A],
\end{aligned}
\label{eq:a-phi}
\end{equation}
 and therefore 
\begin{equation}
\left|\varphi(A)\cap\left([n+m]\setminus F(A)\right)\right|=\left|\varphi_{1}[A]\right|=\left|A\right|.\label{eq:a-phi-size}
\end{equation}
 Note that $\left|F(A)\right|\overset{\eqref{eq:e-x0}}{=}m$. The elements of $F(A)$ are the only elements of $[n+m]$ such
that \eqref{eq:a-phi} does not determine whether they are elements of $\varphi(A)$. In particular, 
\begin{equation}
F(A)\subseteq[n+m]\setminus\varphi_{1}[A].\label{eq:F}
\end{equation}
 From \eqref{eq:levels} and \eqref{eq:a-phi-size}, $\varphi(A)$ contains no element of $F(A)$ if $\left|A\right|\le k-1$,
exactly one if $k\le\left|A\right|\le k+1$, and all elements of $F(A)$ if $k+2\le\left|A\right|$. For $A\in\binom{[n]}{k}\cup\binom{[n]}{k+1}$,
let $f(A)$ be the single element of $\varphi(A)\cap F(A)$.
\begin{claim}
\label{claim:y}One of the following holds:

\gdef\labelwidthi{\widthof{\textbf{\textup{A0. }}}}
\gdef\labeli{\textbf{\textup{A\arabic{enumi}. }}}
\gdef\propertyref#1{A#1}
\gdef\refi{\propertyref{\arabic{enumi}}}
\begin{enumerate}[label=\labeli, ref=\refi, labelsep=0em, leftmargin=0em, labelwidth=\labelwidthi, itemindent=\labelwidth, align=left]
\item \label{enu:X0}There is a $y\in X_{\emptyset}\cup X_{2}$ such that, for every $A\in\binom{[n]}{k}$, we have $f(A)=y$.
(In fact in this case $y\in X_{\emptyset}$, since for a $y\in X_{2}$ and $A\in\binom{[n]\setminus\{\varphi_{2}^{-1}(y)\}}{k}$
we would have $y\notin F(A)$. We do not use this.)
\item \label{enu:X1}There is a $y\in X_{1}$ such that, for every $A\in\binom{[n]\setminus\{\varphi_{1}^{-1}(y)\}}{k}$, we have
$f(A)=y$. (Note that when $\varphi_{1}^{-1}(y)\notin A$, $y\in F(A)$ holds by the definition of $F(A)$.)
\end{enumerate}
\end{claim}

First we show that \prettyref{prop:noninduced-lower-general} follows from this claim. We use the constant-weight code $\C\subset\B$
given by \prettyref{lem:modp code}. If \ref{enu:X0} holds in \prettyref{claim:y}, then we use the statement~\eqref{eq:modp code statement}
in \prettyref{lem:modp code} with the same $n$~and~$m$ as in \eqref{prop:noninduced-lower-general}, $X_{2}\cup X_{\emptyset}$
in the place of~$Y$, and $y$ as given by \prettyref{claim:y}. There is a set $C\in\binom{X_{12}\cup X_{1}}{k}$ such
that $C\cup\{y\}\in\C\subset\B$. Then $\varphi_{1}^{-1}[C]\in\Q_{n}$, and $\varphi\left(\varphi_{1}^{-1}[C]\right)=C\cup\{y\}\in\B$,
contradicting that the image of~$\varphi$ is in~$\R$. If \ref{enu:X1} holds in \prettyref{claim:y}, then we use the
statement~\eqref{eq:modp code statement} with $n-1$ in the place of $n$, $m+1$ in the place of $m$, $X_{2}\cup X_{\emptyset}\cup\{y\}$
in the place of $Y$, and $y$ as given by \prettyref{claim:y}. There is a set $C\in\binom{(X_{12}\cup X_{1})\setminus\{y\}}{k}$
such that $C\cup\{y\}\in\C\subset\B$. Then $\varphi_{1}^{-1}[C]\in\binom{[n]\setminus\{\varphi_{1}^{-1}(y)\}}{k}\subset\Q_{n}$,
and $\varphi\left(\varphi_{1}^{-1}[C]\right)=C\cup\{y\}\in\B$, contradicting that the image of~$\varphi$ is in~$\R$.

To prove \prettyref{claim:y}, we need the following.
\begin{claim}
\label{obs:neighbors}If $A,B\in\binom{[n]}{k}$ with a symmetric difference of size~$2$, and $f(A)\ne f(B)$, then at
least one of the following holds:
\begin{itemize}
\item $f(A)=\varphi_{1}(b)$ where $\{b\}=B\setminus A$. (This implies $b\in E$ and $f(A)\in X_{1}$.)
\item $f(B)=\varphi_{1}(a)$ where $\{a\}=A\setminus B$. (This implies $a\in E$ and $f(B)\in X_{1}$.)
\end{itemize}
\begin{proof}
Indeed, $\left|A\cup B\right|=k+1$, so 
\begin{equation}
\varphi(A\cup B)=\varphi_{1}[A\cup B]\cup\{f(A\cup B)\}=\varphi_{1}[A\cap B]\cup\{\varphi_{1}(a),\varphi_{1}(b),f(A\cup B)\}.\label{eq:a-b-phi}
\end{equation}
Furthermore, 
\begin{align}
\varphi(A\cup B) & \supset\varphi(A)=\varphi_{1}[A]\cup\{f(A)\}=\varphi_{1}[A\cap B]\cup\{\varphi_{1}(a),f(A)\}\tandfinal\label{eq:phi-a}\\
\varphi(A\cup B) & \supset\varphi(B)=\varphi_{1}[B]\cup\{f(B)\}=\varphi_{1}[A\cap B]\cup\{\varphi_{1}(b),f(B)\}.\label{eq:phi-b}
\end{align}
By \eqref{eq:a-b-phi}, \eqref{eq:phi-a} and~\eqref{eq:phi-b}, we have {\thickmuskip=5mu plus 3mu minus 3mu \medmuskip=2mu $\left|\{\varphi_{1}(a),f(A),\varphi_{1}(b),f(B)\}\right|\le\left|\{\varphi_{1}(a),\varphi_{1}(b),f(A\cup B)\}\right|=\nobreak3$},
therefore the elements on the left-hand side of the inequality are not distinct. We know that $\varphi_{1}$ is an injection,
and by~\eqref{eq:F}, $f(S)\notin\varphi_{1}[S]$ for any $S$ of size $k$ or $k+1$. We have assumed $f(A)\ne f(B)$.
It follows that $f(A)=\varphi_{1}(b)$ or $f(B)=\varphi_{1}(a)$. This completes the proof of \prettyref{obs:neighbors}.
\end{proof}
\end{claim}

\present{grey}We first prove \prettyref{claim:y} under the condition $m\le n-\sqrt{8(n+m)-15}-1$, as the proof is simpler
than the proof for arbitrary $m$\emph{.} (For large $n$, this condition holds whenever the ratio of $m$ and $n$ is not
very close to 1.)
\begin{boldproof}[Proof of \prettyref{claim:y} when \thinmuskip=2mu \medmuskip=3mu \thickmuskip=4mu $m\le n-\sqrt{8(n+m)-15}-1$]
At the beginning of the proof of \prettyref{prop:noninduced-lower-general}, we chose an arbitrary $k$ between {\medmuskip=3mu plus 3mu minus 3mu $\sqrt{8(n+m)-15}$
and $n-1-\sqrt{8(n+m)-15}$}. Now we will assume that $k\le n-m$; this is satisfied by choosing e.g. $k=\sqrt{8(n+m)-15}$.
Then, by~\eqref{eq:D ge n-m}, there exists an $B\in\binom{[n]}{k}$ such that $B\subseteq D$. We show that \prettyref{claim:y}
holds with $y=f(B)$. In fact, since $B\cap E=\emptyset$, we have $f(B)\in F(B)=X_{\emptyset}$, and we show that $f(A)=f(B)$
for every $A\in\binom{[n]}{k}$.

Take an $A\in\binom{[n]}{k}$. We can get from $B$ to $A$ by replacing one element at a time, in such a way that we never
add an element that is not an element of $A$, and we never remove an element of $A$ (whether it is also an element of $B$,
or we have added it). In particular, we never remove an element of $E$. That is, there is a sequence $B=B_{0},B_{1},\ldots,B_{l}=A$
such that $\left|B_{i}\triangle B_{i+1}\right|=2$ and $b_{i}^{\leftarrow}\in D$ where $\{b_{i}^{\leftarrow}\}=B_{i}\setminus B_{i+1}$.
Let $b_{i+1}^{\rightarrow}\in[n]$ such that $\{b_{i+1}^{\rightarrow}\}=B_{i+1}\setminus B_{i}$.

We show by induction that $f(B_{i})=f(B)$ for every $i=0,\ldots,l$. Assume that $f(B)=f(B_{i})\ne f(B_{i+1})$. By \prettyref{obs:neighbors},
either $f(B_{i})=\varphi_{1}(b_{i+1}^{\rightarrow})$ or $f(B_{i+1})=\varphi_{1}(b_{i}^{\leftarrow})$. The former implies
$f(B_{i})=f(B)\in X_{1}$, contradicting that it is in $X_{\emptyset}$. The latter implies $b_{i}^{\leftarrow}\in E$, contradicting
that it is in $D$.
\end{boldproof}
\begin{boldproof}[Proof of \prettyref{claim:y}]
The general form of \prettyref{claim:y} will be a consequence of the following lemma.
\begin{lem}
\label{lem:y inductive}Let $n,l\in\naturals$ such that $n\ge5$ and $1\le l\le n-3$, and let $X$ and $Y$ be disjoint
sets such that $\left|X\right|=n$. Let $g:\binom{X}{l}\rightarrow X\cup Y$ be a function such that for every $A\in\binom{X}{l}$,
$g(A)\notin A$; and for every $A,B\in\binom{X}{l}$ with a symmetric difference of size~$2$, where $g(A)\ne g(B)$, at
least one of $\{g(A)\}=B\setminus A$ and $\{g(B)\}=A\setminus B$ holds. Then there is a $y\in X\cup Y$ such that $g(A)=y$
for every $A\in\binom{X\setminus\{y\}}{l}$.
\end{lem}

\begin{subproof}
We prove \prettyref{lem:y inductive} by induction on~$l$.

If $l=1$, the sets are singletons, and every symmetric difference is of size~$2$. We define a graph on~$X$: we connect
two elements $a$ and~$b$ if $g(\{a\})\ne g(\{b\})$. This graph is the complement of a graph whose components are complete
graphs (with the components defined by the values of $a\mapsto g(\{a\})$). For every $a,b\in X$ such that $ab$ is an edge,
we have $g(\{a\})=b$ or $g(\{b\})=a$. Direct the graph such that we have the directed edge $(a,b)$ when $g(\{a\})=b$
(we may direct some edges in both directions).

The out-degree of every vertex is at most~1. Thus the number of edges is at most~$n$. By our assumptions $n\ge4$; the
only graphs with these properties on at least~5 vertices (ignoring the directions of the edges) are the empty graph and
a star on $n$ vertices. If it is an empty graph, then $g(\{a\})$ is the same for every $a\in X$ (and it is necessarily
in~$Y$); the statement of \prettyref{lem:y inductive} holds with $y=g(\{a\})$. If the graph is a star on $n$ vertices,
let $a$ be the center. Since at most one edge is directed outward from~$a$, all but at most one edge is directed towards~$a$.
That is, $g(\{b\})=a$ for all but at most one $b\in X\setminus\{a\}$. Since the leaves of the star are not connected, $g$~has
the same values on them as singletons, so in fact $g(\{b\})=a$ for every $b\in X\setminus\{a\}$, and the lemma holds with
$y=a$.

Now let $l\ge2$. If $g(A)$ is the same for every $A\in\binom{X}{l}$, the lemma holds with that value as~$y$. Assume
that $g(A)$ is not the same for every $A\in\binom{X}{l}$. Since the Johnson graph is connected, there are sets with a symmetric
difference of size~$2$ with different~$g$; consequently there is an~$a\in X$ such that there are sets containing~$a$
with different~$g$.

We use the induction hypothesis with $\tilde{l}=l-1$, $\tilde{n}=n-1$, $\tilde{X}=X\setminus\{a\}$, $\tilde{g}\bigl(\tilde{A}\bigr)=g\bigl(\{a\}\cup\tilde{A}\bigr)$
for $\tilde{A}\in\binom{\tilde{X}}{l-1}$, and $Y$~unchanged. Note that since $g\bigl(\{a\}\cup\tilde{A}\bigr)\notin\{a\}\cup\tilde{A}$,
in fact $\tilde{g}\bigl(\tilde{A}\bigr)\in\tilde{X}\cup Y$ and $\tilde{g}\bigl(\tilde{A}\bigr)\notin\tilde{A}$, so the
conditions of the induction hypothesis hold. So there is a $b\in\tilde{X}\cup Y$ such that $g\bigl(\tilde{A}\bigr)=b$ for
every $\tilde{A}\in\binom{\tilde{X}\setminus\{b\}}{l}$; equivalently, $b\in(X\cup Y)\setminus\{a\}$ such that $g(A)=b$
for every $A\in\binom{X}{l}$ that contains $a$ but not~$b$. If $b$ were in~$Y$, then $g(A)$ would be the same for
every $A\in\binom{X}{l}$ that contains $a$, contradicting our assumption. So $b\in X$. 
\begin{figure}
\centering{}\sidebyside{\sidebysidebox{\begin{center}
\begin{tikzpicture}[line cap=round,line join=round,>=triangle 45,x=1.0cm,y=1.0cm]
\draw [line width=0.8pt] (0.,0.)-- (2.,0.);
\draw [line width=0.8pt] (2.,0.)-- (2.618033988749895,1.9021130325903064);
\draw [line width=0.8pt] (2.618033988749895,1.9021130325903064)-- (1.,3.077683537175253);
\draw [line width=0.8pt] (1.,3.077683537175253)-- (-0.6180339887498947,1.9021130325903073);
\draw [line width=0.8pt] (-0.6180339887498947,1.9021130325903073)-- (0.,0.);
\draw [line width=0.8pt] (0.,0.)-- (2.618033988749895,1.9021130325903064);
\draw [line width=0.8pt] (2.618033988749895,1.9021130325903064)-- (-0.6180339887498947,1.9021130325903073);
\draw [line width=0.8pt] (-0.6180339887498947,1.9021130325903073)-- (2.,0.);
\draw [line width=0.8pt] (2.,0.)-- (1.,3.077683537175253);
\draw [line width=0.8pt] (1.,3.077683537175253)-- (0.,0.);
\begin{scriptsize}
\draw [fill=black] (0.,0.) circle (2.5pt);
\draw[color=black] (0.08,-0.25) node {$a$};
\draw [fill=black] (2.,0.) circle (2.5pt);
\draw[color=black] (2.18,-0.21) node {$b$};
\draw [fill=black] (2.618033988749895,1.9021130325903064) circle (2.5pt);
\draw[color=black] (2.76,2.15) node {$c$};
\draw [fill=black] (1.,3.077683537175253) circle (2.5pt);
\draw[color=black] (1.1,3.39) node {$d$};
\draw [fill=black] (-0.6180339887498947,1.9021130325903073) circle (2.5pt);
\draw[color=black] (-0.84,2.19) node {$e$};
\draw[color=black] (-0.5,0.81) node {$b$};
\draw[color=black] (1.22,1.11) node {$b$};
\draw[color=black] (0.66,1.47) node {$b$};
\end{scriptsize}
\end{tikzpicture}\caption{\label{fig:g(au)=00003Db}$n=5$, $l=2$. The label on an edge $uv$ is $g(\{u,v\})$. (The elements of $Y$ are not shown
in the figure.)}
\par\end{center}}\sidebysidebox{\begin{center}
\begin{tikzpicture}[line cap=round,line join=round,>=triangle 45,x=1.0cm,y=1.0cm]
\draw [line width=0.8pt] (0.,0.)-- (2.,0.);
\draw [line width=0.8pt] (2.,0.)-- (2.618033988749895,1.9021130325903064);
\draw [line width=0.8pt] (2.618033988749895,1.9021130325903064)-- (1.,3.077683537175253);
\draw [line width=0.8pt] (1.,3.077683537175253)-- (-0.6180339887498947,1.9021130325903073);
\draw [line width=0.8pt] (-0.6180339887498947,1.9021130325903073)-- (0.,0.);
\draw [line width=0.8pt] (0.,0.)-- (2.618033988749895,1.9021130325903064);
\draw [line width=0.8pt] (2.618033988749895,1.9021130325903064)-- (-0.6180339887498947,1.9021130325903073);
\draw [line width=0.8pt] (-0.6180339887498947,1.9021130325903073)-- (2.,0.);
\draw [line width=0.8pt] (2.,0.)-- (1.,3.077683537175253);
\draw [line width=0.8pt] (1.,3.077683537175253)-- (0.,0.);
\begin{scriptsize}
\draw [fill=black] (0.,0.) circle (2.5pt);
\draw[color=black] (0.08,-0.25) node {$a$};
\draw [fill=black] (2.,0.) circle (2.5pt);
\draw[color=black] (2.18,-0.21) node {$b$};
\draw [fill=black] (2.618033988749895,1.9021130325903064) circle (2.5pt);
\draw[color=black] (2.76,2.15) node {$c$};
\draw [fill=black] (1.,3.077683537175253) circle (2.5pt);
\draw[color=black] (1.1,3.39) node {$d$};
\draw [fill=black] (-0.6180339887498947,1.9021130325903073) circle (2.5pt);
\draw[color=black] (-0.84,2.19) node {$e$};
\draw[color=black] (1.94,2.61) node {$\bullet$};
\draw[color=black] (0.12,2.61) node {$\bullet$};
\draw[color=black] (-0.5,0.81) node {$b$};
\draw[color=black] (1.22,1.11) node {$b$};
\draw[color=black] (1.02,1.76) node {$\bullet$};
\draw[color=black] (0.66,1.47) node {$b$};
\end{scriptsize}
\end{tikzpicture}\caption{\label{fig:bullets}We first prove that the marked pairs are assigned either $a$ or $b$, then that they are all assigned
the same value.}
\par\end{center}}}
\end{figure}
(See \prettyref{fig:g(au)=00003Db}.)

Take a $C\in\binom{X\setminus\{a,b\}}{l}$. We show that $g(C)\in\{a,b\}$. Take an arbitrary $c\in C$. $(C\setminus\{c\})\cup\{a\}$
contains $a$ but not $b$, so $g\bigl((C\setminus\{c\})\cup\{a\}\bigr)=b$. Since $\left|C\triangle\bigl((C\setminus\{c\})\cup\{a\}\bigr)\right|=2$,
either $g(C)=g\bigl((C\setminus\{c\})\cup\{a\}\bigr)=b$, or $g(C)=a$, or $g\bigl((C\setminus\{c\})\cup\{a\}\bigr)=c$ ---
but the last option is false.

Now we show that $g(C)$ is the same for every $C\in\binom{X\setminus\{a,b\}}{l}$. (See \prettyref{fig:bullets}.) If this
is not the case, there are $C,D\in\binom{X\setminus\{a,b\}}{l}$ such that $\left|C\triangle D\right|=2$, and $g(C)=a$
but $g(D)=b$. But this implies that $C\setminus D=\{b\}$ or $D\setminus C=\{a\}$, which is impossible because $a,b\notin C,D$.

We already know that $g(A)=b$ for every $A\in\binom{X}{l}$ that contains $a$ but not~$b$. If $g(C)=b$ for every $C\in\binom{X\setminus\{a,b\}}{l}$,
then the lemma holds with $y=b$ (see \prettyref{fig:y=00003Db}). So assume instead that $g(C)=a$ for every $C\in\binom{X\setminus\{a,b\}}{l}$
(\prettyref{fig:assumption}).
\begin{figure}
\sidebyside{\sidebysidebox{\begin{center}
\begin{tikzpicture}[line cap=round,line join=round,>=triangle 45,x=1.0cm,y=1.0cm]
\draw [line width=0.8pt] (0.,0.)-- (2.,0.);
\draw [line width=0.8pt] (2.,0.)-- (2.618033988749895,1.9021130325903064);
\draw [line width=1.6pt] (2.618033988749895,1.9021130325903064)-- (1.,3.077683537175253);
\draw [line width=1.6pt] (1.,3.077683537175253)-- (-0.6180339887498947,1.9021130325903073);
\draw [line width=1.6pt] (-0.6180339887498947,1.9021130325903073)-- (0.,0.);
\draw [line width=1.6pt] (0.,0.)-- (2.618033988749895,1.9021130325903064);
\draw [line width=1.6pt] (2.618033988749895,1.9021130325903064)-- (-0.6180339887498947,1.9021130325903073);
\draw [line width=0.8pt] (-0.6180339887498947,1.9021130325903073)-- (2.,0.);
\draw [line width=0.8pt] (2.,0.)-- (1.,3.077683537175253);
\draw [line width=1.6pt] (1.,3.077683537175253)-- (0.,0.);
\begin{scriptsize}
\draw [fill=black] (0.,0.) circle (2.5pt);
\draw[color=black] (0.08,-0.25) node {$a$};
\draw [fill=black] (2.,0.) circle (2.5pt);
\draw[color=black] (2.18,-0.21) node {$b$};
\draw [fill=black] (2.618033988749895,1.9021130325903064) circle (2.5pt);
\draw[color=black] (2.76,2.15) node {$c$};
\draw [fill=black] (1.,3.077683537175253) circle (2.5pt);
\draw[color=black] (1.1,3.39) node {$d$};
\draw [fill=black] (-0.6180339887498947,1.9021130325903073) circle (2.5pt);
\draw[color=black] (-0.84,2.19) node {$e$};
\draw[color=black] (2,2.59) node {$\boldsymbol{b}$};
\draw[color=black] (-0.04,2.59) node {$\boldsymbol{b}$};
\draw[color=black] (-0.5,0.81) node {$\boldsymbol{b}$};
\draw[color=black] (1.22,1.11) node {$\boldsymbol{b}$};
\draw[color=black] (0.98,1.71) node {$\boldsymbol{b}$};
\draw[color=black] (0.66,1.47) node {$\boldsymbol{b}$};
\end{scriptsize}
\end{tikzpicture}
\par\end{center}
\begin{center}
\caption{\label{fig:y=00003Db}In this case, \prettyref{lem:y inductive} holds with $y=b$.}
\par\end{center}}\sidebysidebox{\begin{center}
\begin{tikzpicture}[line cap=round,line join=round,>=triangle 45,x=1.0cm,y=1.0cm]
\draw [line width=0.8pt] (0.,0.)-- (2.,0.);
\draw [line width=0.8pt] (2.,0.)-- (2.618033988749895,1.9021130325903064);
\draw [line width=0.8pt] (2.618033988749895,1.9021130325903064)-- (1.,3.077683537175253);
\draw [line width=0.8pt] (1.,3.077683537175253)-- (-0.6180339887498947,1.9021130325903073);
\draw [line width=0.8pt] (-0.6180339887498947,1.9021130325903073)-- (0.,0.);
\draw [line width=0.8pt] (0.,0.)-- (2.618033988749895,1.9021130325903064);
\draw [line width=0.8pt] (2.618033988749895,1.9021130325903064)-- (-0.6180339887498947,1.9021130325903073);
\draw [line width=0.8pt] (-0.6180339887498947,1.9021130325903073)-- (2.,0.);
\draw [line width=0.8pt] (2.,0.)-- (1.,3.077683537175253);
\draw [line width=0.8pt] (1.,3.077683537175253)-- (0.,0.);
\begin{scriptsize}
\draw [fill=black] (0.,0.) circle (2.5pt);
\draw[color=black] (0.08,-0.25) node {$a$};
\draw [fill=black] (2.,0.) circle (2.5pt);
\draw[color=black] (2.18,-0.21) node {$b$};
\draw [fill=black] (2.618033988749895,1.9021130325903064) circle (2.5pt);
\draw[color=black] (2.76,2.15) node {$c$};
\draw [fill=black] (1.,3.077683537175253) circle (2.5pt);
\draw[color=black] (1.1,3.39) node {$d$};
\draw [fill=black] (-0.6180339887498947,1.9021130325903073) circle (2.5pt);
\draw[color=black] (-0.84,2.19) node {$e$};
\draw[color=black] (2,2.59) node {$a$};
\draw[color=black] (-0.04,2.59) node {$a$};
\draw[color=black] (-0.5,0.81) node {$b$};
\draw[color=black] (1.22,1.11) node {$b$};
\draw[color=black] (0.98,1.71) node {$a$};
\draw[color=black] (0.66,1.47) node {$b$};
\end{scriptsize}
\end{tikzpicture}
\par\end{center}
\begin{center}
\caption{\label{fig:assumption}We assume that the pairs marked in \prettyref{fig:bullets} are assigned $a$ instead.}
\par\end{center}}}
\end{figure}

Let $B\in\binom{X}{l}$ such that it contains $b$ but not~$a$. We show that $g(B)=a$. Take two different, arbitrary elements
$c,d\in X\setminus(B\cup\{a\})$.
\begin{figure}
\sidebyside{\sidebysidebox{\begin{center}
\begin{tikzpicture}[line cap=round,line join=round,>=triangle 45,x=1.0cm,y=1.0cm]
\draw [line width=0.8pt] (0.,0.)-- (2.,0.);
\draw [line width=0.8pt] (2.,0.)-- (2.618033988749895,1.9021130325903064);
\draw [line width=0.8pt] (2.618033988749895,1.9021130325903064)-- (1.,3.077683537175253);
\draw [line width=0.8pt] (1.,3.077683537175253)-- (-0.6180339887498947,1.9021130325903073);
\draw [line width=0.8pt] (-0.6180339887498947,1.9021130325903073)-- (0.,0.);
\draw [line width=0.8pt] (0.,0.)-- (2.618033988749895,1.9021130325903064);
\draw [line width=0.8pt] (2.618033988749895,1.9021130325903064)-- (-0.6180339887498947,1.9021130325903073);
\draw [line width=0.8pt] (-0.6180339887498947,1.9021130325903073)-- (2.,0.);
\draw [line width=0.8pt] (2.,0.)-- (1.,3.077683537175253);
\draw [line width=0.8pt] (1.,3.077683537175253)-- (0.,0.);
\begin{scriptsize}
\draw [fill=black] (0.,0.) circle (2.5pt);
\draw[color=black] (0.08,-0.25) node {$a$};
\draw [fill=black] (2.,0.) circle (2.5pt);
\draw[color=black] (2.18,-0.21) node {$b$};
\draw [fill=black] (2.618033988749895,1.9021130325903064) circle (2.5pt);
\draw[color=black] (2.76,2.15) node {$c$};
\draw [fill=black] (1.,3.077683537175253) circle (2.5pt);
\draw[color=black] (1.1,3.39) node {$d$};
\draw [fill=black] (-0.6180339887498947,1.9021130325903073) circle (2.5pt);
\draw[color=black] (-0.84,2.19) node {$e$};
\draw[color=black] (2,2.59) node {$a$};
\draw[color=black] (-0.04,2.59) node {$a$};
\draw[color=black] (-0.5,0.81) node {$b$};
\draw[color=black] (1.22,1.11) node {$b$};
\draw[color=black] (0.98,1.71) node {$a$};
\draw[color=black] (0.8,1.07) node {$\bullet$};
\draw[color=gray] (1.32,1.53) node {$\bullet$};
\draw[color=gray] (2.46,0.86) node {$\bullet$};
\draw[color=black] (0.66,1.47) node {$b$};
\end{scriptsize}
\end{tikzpicture}\caption{\label{fig:g(B)=00003Da}We show that \thickmuskip=3mu plus 5mu \relax $g(B)=a$ for $B=\nobreak\{b,e\}$, for arbitrary
choice of $e\protect\ne a,b$.}
\par\end{center}}\sidebysidebox{\begin{center}
\begin{tikzpicture}[line cap=round,line join=round,>=triangle 45,x=1.0cm,y=1.0cm]
\draw [line width=0.8pt] (0.,0.)-- (2.,0.);
\draw [line width=1.6pt] (2.,0.)-- (2.618033988749895,1.9021130325903064);
\draw [line width=1.6pt] (2.618033988749895,1.9021130325903064)-- (1.,3.077683537175253);
\draw [line width=1.6pt] (1.,3.077683537175253)-- (-0.6180339887498947,1.9021130325903073);
\draw [line width=0.8pt] (-0.6180339887498947,1.9021130325903073)-- (0.,0.);
\draw [line width=0.8pt] (0.,0.)-- (2.618033988749895,1.9021130325903064);
\draw [line width=1.6pt] (2.618033988749895,1.9021130325903064)-- (-0.6180339887498947,1.9021130325903073);
\draw [line width=1.6pt] (-0.6180339887498947,1.9021130325903073)-- (2.,0.);
\draw [line width=1.6pt] (2.,0.)-- (1.,3.077683537175253);
\draw [line width=0.8pt] (1.,3.077683537175253)-- (0.,0.);
\begin{scriptsize}
\draw [fill=black] (0.,0.) circle (2.5pt);
\draw[color=black] (0.08,-0.25) node {$a$};
\draw [fill=black] (2.,0.) circle (2.5pt);
\draw[color=black] (2.18,-0.21) node {$b$};
\draw [fill=black] (2.618033988749895,1.9021130325903064) circle (2.5pt);
\draw[color=black] (2.76,2.15) node {$c$};
\draw [fill=black] (1.,3.077683537175253) circle (2.5pt);
\draw[color=black] (1.1,3.39) node {$d$};
\draw [fill=black] (-0.6180339887498947,1.9021130325903073) circle (2.5pt);
\draw[color=black] (-0.84,2.19) node {$e$};
\draw[color=black] (2.52,0.81) node {$\boldsymbol{a}$};
\draw[color=black] (2,2.59) node {$\boldsymbol{a}$};
\draw[color=black] (-0.04,2.59) node {$\boldsymbol{a}$};
\draw[color=black] (-0.5,0.81) node {$b$};
\draw[color=black] (1.22,1.11) node {$b$};
\draw[color=black] (0.98,1.71) node {$\boldsymbol{a}$};
\draw[color=black] (0.76,1.09) node {$\boldsymbol{a}$};
\draw[color=black] (1.32,1.53) node {$\boldsymbol{a}$};
\draw[color=black] (0.66,1.47) node {$b$};
\end{scriptsize}
\end{tikzpicture}\caption{\label{fig:y=00003Da}\prettyref{lem:y inductive} holds with $y=a$.}
\par\end{center}}}
\end{figure}
 (There are at least two such elements because $l\le n-3$. See \prettyref{fig:g(B)=00003Da}.) Since $\left|B\triangle\bigl((B\setminus\{b\})\cup\{c\}\bigr)\right|=2$,
either $g(B)=g\bigl((B\setminus\{b\})\cup\{c\}\bigr)=a$, or $g(B)=c$, or $g\bigl((B\setminus\{b\})\cup\{c\}\bigr)=b$ ---
but the last option is false. So if $g(B)\ne a$, then $g(B)=c$. By the same reasoning applied with $d$ in the place of
$c$, if $g(B)\ne a$, then $g(B)=d$, a contradiction. So $g(B)=a$ for every $B\in\binom{X}{l}$ that contains $b$ but
not~$a$. Since we already know that $g(C)=a$ for every $C\in\binom{X\setminus\{a,b\}}{l}$, this implies that the lemma
holds with $y=a$ (see \prettyref{fig:y=00003Da}). This completes the proof of \prettyref{lem:y inductive}.
\end{subproof}
Using \prettyref{lem:y inductive}, we show \prettyref{claim:y}. Let $l=k$, $X=X_{12}\cup X_{1}$, $Y=X_{\emptyset}\cup X_{2}$,
and for a $B\in\binom{X_{12}\cup X_{1}}{k}$, let $g(B)=f\left(\varphi_{1}^{-1}[B]\right)$. (Since $\varphi_{1}$~is an
injection and its image is $X_{12}\cup X_{1}$, we have $\varphi_{1}^{-1}[B]\in\binom{[n]}{k}=\Dom f$.) The conditions of
\prettyref{lem:y inductive} hold by \prettyref{obs:neighbors}. By \prettyref{lem:y inductive}, there is a $y\in[n+m]$
such that $f\left(\varphi_{1}^{-1}[B]\right)=g(B)=y$ for every $B\in\binom{(X_{12}\cup X_{1})\setminus\{y\}}{k}$ (where
$(X_{12}\cup X_{1})\setminus\{y\}$ may coincide with $X_{12}\cup X_{1}$).

If $y\in X_{\emptyset}\cup X_{2}$, then for every $A\in\binom{[n]}{k}$, we have $\varphi_{1}[A]\in\binom{(X_{12}\cup X_{1})\setminus\{y\}}{k}=\binom{X_{12}\cup X_{1}}{k}$,
and $f(A)=f\left(\varphi_{1}^{-1}\left[\varphi_{1}[A]\right]\right)=y$, so \ref{enu:X0} holds in \prettyref{claim:y}.
If $y\in X_{12}\cup X_{1}$, then for every $A\in\binom{[n]\setminus\{\varphi_{1}^{-1}(y)\}}{k}$, we have $\varphi_{1}[A]\in\binom{(X_{12}\cup X_{1})\setminus\{y\}}{k}$,
and $f(A)=f\left(\varphi_{1}^{-1}\left[\varphi_{1}[A]\right]\right)=y$. Since $f(A)\in F(A)\subseteq X_{1}\cup X_{2}\cup X_{\emptyset}$,
we also have $y\in X_{1}$, so \ref{enu:X1} holds in \prettyref{claim:y}.
\end{boldproof}

\subsection{\label{subsec:cox stolee}A probabilistic construction showing $R_{w}(\protect\Q_{m},\protect\Q_{n})\ge m+n+1$ when $m\ge3$}
\begin{thm}
\label{prop:CS-generalized}If $n,m\in\naturals$, $n$~is sufficiently large, and $m\ge3$, then there exist $\B,\R\subset\Q_{n+m}$
such that $\B\sqcup\R=\Q_{n+m}$, $\Q_{m}$~is not a weak subposet of~$\B$, and $\Q_{n}$~is not a weak subposet of~$\R$.
\end{thm}

In most of this subsection, we prove \prettyref{prop:CS-generalized}. The core of the random construction will be in \prettyref{lem:random}.
In the proof of \prettyref{lem:random} we will use the asymmetric version of the Lov\'asz Local Lemma.
\begin{lem}[Asymmetric Lov\'asz Local Lemma]
\label{lem:asymlocal}Let $\A$ be a collection of events. For $A\in\A$, let $\Gamma(A)$~be the set of those events in~$\A$,
other than $A$ itself, that are not independent of~$A$. If there is a function $x:\A\to[0,1)$ such that for every $A\in\A$,
we have 
\begin{equation}
P(A)\le x(A)\prod_{B\in\Gamma(A)}(1-x(B)),\label{eq:asymlocal-ineq}
\end{equation}
then there is a non-zero probability that none of the events occur.
\end{lem}

\begin{claim}
\label{lem:random}If $n,m\in\naturals$, $n$~is sufficiently large, and $3\le m\le n$, then there is a family of sets
$\F\subset\binom{[n+m]}{m}$ such that

\gdef\labelwidthi{\widthof{(ii)}}
\gdef\labeli{\textup{(\roman{enumi})}}
\gdef\refi{(\textup{\roman{enumi})}}
\begin{enumerate}[label=\labeli, ref=\refi, labelwidth=\labelwidthi, itemindent=\labelwidth, align=left]
\item \label{enu:supersets}for each $S\in\binom{[n+m]}{m-1}$, $\F$~contains at least~$2$ supersets of~$S$, and
\item \label{enu:subsets}for each $T\in\binom{[n+m]}{m+1}$, $\F$~contains at most $m-1$ subsets of~$T$.
\end{enumerate}
\end{claim}

\begin{proof}
Let $p=\left(4(m+1)\left(n^{2}-1\right)e\right)^{-1/m}$. Let $\F$ be a collection of sets given by taking each set $F\in\binom{[n+m]}{m}$
independently at random with probability~$p$.

{} For any $S\in\binom{[n+m]}{m-1}$, let $A_{S}$ be the event in which $\F$ contains at most~$1$ superset of~$S$, and
for any $T\in\binom{[n+m]}{m+1}$, let $B_{T}$ be the event in which $\F$ contains at least~$m$ subsets of~$T$. We
have
\begin{align*}
P(A_{S}) & =(n+1)(1-p)^{n}p+(1-p)^{n+1}\tandfinal\\
P(B_{T}) & =(m+1)p^{m}(1-p)+p^{m+1}.
\end{align*}

A given event $A_{S}$ is independent of an event of the form $B_{T}$ unless there is a set $F\in\binom{[n+m]}{m}$ such
that $S\subset F\subset T$, i.e., if $S\subset T$. There are $\frac{(n+1)n}{2}$ such events $B_{T}$. $A_{S}$~is not
independent of another event $A_{S'}$ if there is a $F\in\binom{[n+m]}{m}$ such that $S,S'\subset F$, i.e., if the symmetric
difference of $S$ and $S'$ is of size~$2$. There are $(m-1)(n+1)$ such events $A_{S'}$. By symmetry, a given event
$B_{T}$ is independent of all but $\frac{(m+1)m}{2}$ events~$A_{S}$, and is independent of all but $(n-1)(m+1)$ other
events of the form $B_{T'}$.

We want to use \prettyref{lem:asymlocal} to prove that there is a non-zero probability that none of the events $A_{S}$
and $B_{T}$ occur, and thus $\F$ fulfills the conditions of \prettyref{lem:random}. Define a function 
\[
x:\left\{ A_{S}:S\in\binom{[n+m]}{m-1}\right\} \cup\left\{ B_{T}:T\in\binom{[n+m]}{m+1}\right\} \to[0,1)\text{ by}
\]
\[
x(E)=\begin{cases}
y\coloneqq\frac{1}{4(m-1)(n+1)} & \tif E=A_{S}\text{ for some }S\in\binom{[n+m]}{m-1},\\
z\coloneqq\frac{1}{4(n-1)(n+1)} & \tif E=B_{T}\text{ for some }T\in\binom{[n+m]}{m+1}.
\end{cases}
\]
 For an event~$E$ in the domain of~$x$, let $\Gamma(E)$ be the set of other events that are not independent of~$E$.
We will use the bounds 
\begin{equation}
e^{-x}\ge1-x\ge e^{-2x},\label{eq:expbound}
\end{equation}
 which hold when $0\le x\le\frac{1}{2}.$ For any set $T\in\binom{[n+m]}{m+1}$, we have 
\begin{gather*}
x(B_{T})\prod_{E'\in\Gamma(B_{T})}(1-x(E'))=z\overbrace{(1-y)^{(m+1)m/2}}^{E'=A_{S}}\overbrace{(1-z)^{(n-1)(m+1)}}^{E'=B_{T'}}\\
\overset{\eqref{eq:expbound}}{\ge}ze^{-2(y(m+1)m/2+z(n-1)(m+1))}>\frac{1}{4(n-1)(n+1)e}\\
=(m+1)p^{m}>(m+1)p^{m}(1-p)+p^{m+1}=P(B_{T}).
\end{gather*}
 For any set~$S\in\binom{[n+m]}{m-1}$, we have
\begin{gather*}
x(A_{S})\prod_{E'\in\Gamma(A_{S})}(1-x(E'))=y\overbrace{(1-z)^{(n+1)n/2}}^{E'=B_{T}}\overbrace{(1-y)^{(m-1)(n+1)}}^{E'=A_{S'}}\\
\overset{\eqref{eq:expbound}}{\ge}ye^{-2(z(n+1)n/2+y(m-1)(n+1))}>\frac{1}{4(m-1)(n+1)e}\ge\frac{1}{4(n-1)(n+1)e}\text{,\quad and}
\end{gather*}
\begin{equation}
\begin{gathered}P(A_{S})=(n+1)(1-p)^{n}p+(1-p)^{n+1}<\bigl((n+1)p+1\bigr)(1-p)^{n}\\
\overset{\eqref{eq:expbound}}{\le}\bigl((n+1)p+1\bigr)\cdot e^{-pn}\\
<\bigl((n+1)\left(4(m+1)\left(n^{2}-1\right)e\right)^{-1/m}+1\bigr)\cdot e^{-\bigl(4(m+1)e\bigr)^{-1/m}\cdot n^{1-2/m}}.
\end{gathered}
\label{eq:PAS}
\end{equation}
 On the right-hand side of \eqref{eq:PAS}, $\bigl((n+1)\left(4(m+1)\left(n^{2}-1\right)e\right)^{-1/m}+1\bigr)$ is increasing
in $m$ and $e^{-\bigl(4(m+1)e\bigr)^{-1/m}\cdot n^{1-2/m}}$ is decreasing for $m\ge3$. So, by replacing $m$ with $n$
in the first factor, and $m$ with $3$ in the second factor, we have 
\begin{gather*}
P(A_{S})\le\bigl((n+1)\left(4(n+1)\left(n^{2}-1\right)e\right)^{-1/n}+1\bigr)\cdot e^{-\bigl(16e\bigr)^{-1/3}\cdot n^{1/3}}\\
\le\frac{1}{4(n-1)(n+1)e}<x(A_{S})\prod_{E'\in\Gamma(A_{S})}(1-x(E'))
\end{gather*}

when $n$~is sufficiently large. Therefore the function $x$ satisfies the inequality \eqref{eq:asymlocal-ineq} required
by the asymmetric Lov\'asz Local Lemma, so $\F$ has the desired properties.
\end{proof}
Now we are ready to prove \prettyref{prop:CS-generalized} using the family of sets constructed in \prettyref{lem:random}.
We may assume without loss of generality that $m\le n$. Let $\F\subset\binom{[n+m]}{m}$ be the family of sets given by
\prettyref{lem:random}. Let $\B=\binom{[n+m]}{0}\cup\binom{[n+m]}{1}\cup\ldots\cup\binom{[n+m]}{m-2}\cup\F\cup\binom{[n+m]}{m+1}$,
and let $\R=\Q_{n+m}\setminus\B$.

Assume that $\B$ contains a weak copy of~$\Q_{m}$ provided by the injection $\varphi:\Q_{m}\to\B$. Note that $\B$ has
height $m+1$ as a poset. Therefore, for $A\in\Q_{m}$, $\left|\varphi(A)\right|=m$ if $\left|A\right|=m-1$, and $\left|\varphi(A)\right|=m+1$
if $A=[m]$. The $m$ sets of size $m-1$ in~$\Q_{m}$ are mapped to subsets of $\varphi([m])$ in $\F=\binom{[n+m]}{m}\cap\B$.
But, by \ref{enu:subsets} in \prettyref{lem:random}, only at most $m-1$ subsets of $\varphi([m])$ are in $\F$, a contradiction.

Similarly, assume that $\R$ contains a weak copy of~$\Q_{n}$ provided by the injection $\varphi:\Q_{n}\to\R$. Note that
$\R$ has height $n+1$. Therefore, for $A\in\Q_{n}$, $\left|\varphi(A)\right|=m-1$ if $A=\emptyset$, $\left|\varphi(A)\right|=m$
if $\left|A\right|=1$, and $\left|\varphi(A)\right|=\left|A\right|+m$ if $\left|A\right|\in\{2,3,\ldots,n\}$. The $n$
singletons of~$\Q_{n}$ are mapped to supersets of $\varphi(\emptyset)$ in $\binom{[n+m]}{m}\cap\R=\binom{[n+m]}{m}\setminus\F$.
But, by \ref{enu:supersets} in \prettyref{lem:random}, at least~$2$ supersets of $\varphi(\text{\ensuremath{\emptyset}})$
are in~$\F$, so at most~$n-1$ are in $\binom{[n+m]}{m}\setminus\F$, a contradiction.
\begin{remark*}
The above proof of \prettyref{prop:CS-generalized} cannot be easily made to work for $m=2$. More precisely, the following
\MakeLowercase{\crtcrefnamebylabel{claim:cox-stolee-fails}} holds.
\end{remark*}
\begin{claim}
\label{claim:cox-stolee-fails}The conclusion of \prettyref{lem:random} does not hold for $m=2$, and in fact there is no
$\F\subset\binom{[n+2]}{2}$ such that, for $\B=\{\emptyset\}\cup\F\cup\binom{[n+2]}{3}$ and $\R=\Q_{n+2}\setminus\B$,
$\Q_{2}$~is not a subposet of~$\B$, and $\Q_{n}$~is not a subposet of~$\R$.
\end{claim}

\begin{proof}
A family of sets $\F\subset\binom{[n+2]}{2}$ that satisfies the condition \ref{enu:supersets} in \prettyref{lem:random}
contains, for any $S\in\binom{[n+2]}{1}$, a pair of sets $A,B$ such that $S\subset A,B$. Note that $A$ and $B$ have
a symmetric difference of size~$2$. Then $A,B\subset A\cup B\in\binom{[n+2]}{3}$, which contradicts the condition \ref{enu:subsets}
in \prettyref{lem:random}. So the two conditions of \prettyref{lem:random} cannot be satisfied at the same time by a family
of sets $\F\subset\binom{[n+2]}{2}$.

It is easy to check that the conditions \ref{enu:supersets}~and~\ref{enu:subsets} on~$\F$ in \prettyref{lem:random}
are not only sufficient, but also necessary for the above coloring to satisfy the conditions of \prettyref{prop:CS-generalized},
that is, to have no $\Q_{2}$ as a subposet of $\B=\{\emptyset\}\cup\F\cup\binom{[n+2]}{3}$, and no $\Q_{n}$ as a subposet
of $\R=\Q_{n+2}\setminus\B$.
\end{proof}

\section*{Acknowledgments\phantomsection\addcontentsline{toc}{section}{Acknowledgements}}

The second and third authors were supported by the grant IBS-R029-C1. The research of the second author was also partially
supported by the EPSRC, grant no. EP/S00100X/1 (A. Methuku).

\end{document}